\documentclass[12pt,leqno]{article}
\usepackage{amsfonts}
\usepackage{amsmath, amsthm, amsfonts, amssymb, color}
\usepackage{mathrsfs}
\usepackage{color}
\usepackage[notcite,notref]{showkeys}

\newtheorem{thm}{Theorem}[section]
\newtheorem{cor}[thm]{Corollary}
\newtheorem{lem}[thm]{Lemma}
\newtheorem{prp}[thm]{Proposition}

\newtheorem{exa}[thm]{Example}
\newtheorem{rem}[thm]{Remark}
\theoremstyle{definition}

\newcommand{\scr}[1]{\mathscr #1}
\definecolor{wco}{rgb}{0.5,0.2,0.3}

\numberwithin{equation}{section} \theoremstyle{remark}

\newcommand{\ua}{\uparrow}

\title{{\bf Improving   Numerical   Error  Bounds Near Sharp Interface Limit for Stochastic Reaction-Diffusion  Equations }\footnote{This research is supported in
 part by the National Key R\& D Program of China (No. 2022YFA1006000, 2020YFA0712900) and NNSFC (11921001),  the Hong Kong Research Grant Council GRF grant 15302823, NSFC grant 12301526, internal funds (P0039016, P0045336) from Hong Kong Polytechnic University and the CAS AMSS-PolyU Joint Laboratory of Applied Mathematics.} }
\author{{\bf  Jianbo Cui and Feng-Yu Wang   }\\
\footnotesize{Department of Applied Mathematics, The Hong Kong Polytechnic University}\\
\footnotesize{jianbo.cui@polyu.edu.hk}\\
  \footnotesize{ Center for Applied Mathematics, Tianjin University, Tianjin 300072, China}\\
\footnotesize{ wangfy@tju.edu.cn}}

\newcommand{\N}{\mathbb N}
\newcommand{\R}{\mathbb R}
\newcommand{\D}{\mathbb D}

\newcommand{\Z}{\mathbb Z}

\newcommand{\C}{\mathcal C}
\newcommand{\OO}{\mathcal O}

\newcommand{\LL}{\mathcal L}

\newcommand{\EE}{\mathcal E}
\newcommand{\F}{\mathcal F}

\newcommand{\ee}{\varepsilon}

\newcommand{\dd}{\partial}

\newcommand{\ddd}{\displaystyle}

\newcommand{\kk}{\underline{k}}

\newcommand{\<}{\langle}
\renewcommand{\>}{\rangle}
\usepackage{hyperref}

\begin{document}
\allowdisplaybreaks
\def\R{\mathbb R}  \def\ff{\frac} \def\ss{\sqrt} \def\B{\mathbf
B}
\def\N{\mathbb N} \def\kk{\kappa} \def\m{{\bf m}}
\def\ee{\varepsilon}\def\ddd{D^*}
\def\dd{\delta} \def\DD{\Delta} \def\vv{\varepsilon} \def\rr{\rho}
\def\<{\langle} \def\>{\rangle}
  \def\nn{\nabla} \def\pp{\partial} \def\E{\mathbb E}
\def\d{\text{\rm{d}}} \def\bb{\beta} \def\aa{\alpha} \def\D{\scr D}
  \def\si{\sigma} \def\ess{\text{\rm{ess}}}\def\s{{\bf s}}
\def\beg{\begin} \def\beq{\begin{equation}}  \def\F{\scr F}
\def\Ric{\mathcal Ric} \def\Hess{\text{\rm{Hess}}}
\def\e{\text{\rm{e}}} \def\ua{\underline a} \def\OO{\Omega}  \def\oo{\omega}
 \def\tt{\tilde}\def\[{\lfloor} \def\]{\rfloor}
\def\cut{\text{\rm{cut}}} \def\P{\mathbb P} \def\ifn{I_n(f^{\bigotimes n})}
\def\C{\scr C}      \def\aaa{\mathbf{r}}     \def\r{r}
\def\gap{\text{\rm{gap}}} \def\prr{\pi_{{\bf m},\nu}}  \def\r{\mathbf r}
\def\Z{\mathbb Z} \def\vrr{\nu} \def\ll{\lambda}
\def\L{\scr L}\def\Tt{\tt} \def\TT{\tt}\def\II{\mathbb I}
\def\i{{\rm in}}\def\Sect{{\rm Sect}}  \def\H{\mathbb H}
\def\M{\mathbb M}\def\Q{\mathbb Q} \def\texto{\text{o}} \def\LL{\Lambda}
\def\Rank{{\rm Rank}} \def\B{\scr B} \def\i{{\rm i}} \def\HR{\hat{\R}^d}
\def\to{\rightarrow} \def\gg{\gamma}
\def\EE{\scr E} \def\W{\mathbb W}
\def\A{\scr A} \def\Lip{{\rm Lip}}\def\S{\mathbb S}
\def\BB{\scr B}\def\Ent{{\rm Ent}} \def\i{{\rm i}}\def\itparallel{{\it\parallel}}
\def\g{{\mathbf g}}\def\Sect{{\mathcal Sec}}\def\T{\mathcal T}\def\BB{{\bf B}}
\def\f\ell \def\g{\mathbf g}\def\BL{{\bf L}}  \def\BG{{\mathbb G}}
\def\Bd{{D^E}} \def\BdP{D^E_\phi} \def\Bdd{{\bf \dd}} \def\Bs{{\bf s}} \def\GA{\scr A}
\def\Bg{{\bf g}}  \def\Bdd{\psi_B} \def\supp{{\rm supp}}\def\div{{\rm div}}
\def\ddiv{{\rm div}}\def\osc{{\bf osc}}\def\1{{\bf 1}}\def\BD{\mathbb D}
\def\H{{\bf H}}\def\gg{\gamma} \def\n{{\mathbf n}} \def\red{\color{red}}
\maketitle

\begin{abstract}
In the study of geometric surface evolutions, stochastic reaction-diffusion  equation provides a powerful tool for capturing and simulating complex dynamics. A critical challenge in this area is developing numerical approximations that exhibit error bounds with polynomial dependence on $\vv^{-1}$, where the small parameter $\vv>0$ represents  the diffuse interface thickness. The existence of such bounds for fully discrete approximations of stochastic reaction-diffusion  equations remains unclear in the literature. In this work, we address this challenge by leveraging  the asymptotic log-Harnack inequality to overcome the exponential growth of $\vv^{-1}$. Furthermore, we establish the numerical weak error bounds under the truncated Wasserstein distance for the spectral Galerkin method and a  fully discrete  tamed Euler scheme, with explicit polynomial dependence on $\vv^{-1}$.
  \end{abstract} \noindent
 AMS subject Classification:\  60H15, 60H35, 60D05, 65M15.   \\
\noindent
 Keywords:    Stochastic reaction-diffusion equation,  weak convergence, asymptotic log-Harnack inequality, near sharp interface limit.

 \vskip 2cm

 \section{Introduction}

In this paper we investigate numerical approximations for stochastic reaction-diffusion  equations near the sharp  interface limit, i.e. for the following stochastic evolution equation in   $\H:= L^2(\mu)$ over a separable probability space  $(E,\scr B,\mu)$ with small parameter $\vv>0$:
\beq\label{E1} \d u_t +\Big(A u_t+\ff 1 \vv f(u_t)\Big)\d t  =\d W_t,\end{equation}
where $(A,\D(A))$ is a positive definite self-adjoint operator on $\H$, $f$ is a measurable real function on $\R$, $W_t$ is a Wiener process  on $\H$, and  $\vv $ stands for  the thickness of the thin diffuse interface layer.

The equation \eqref{E1} with small $\vv>0$  is a crucial model to   describe the material interfaces during phase transitions   \cite{MR1175626}, and  has attracted growing interest   due to its ability to incorporate uncertainties and randomness arising  from thermal fluctuations, material impurities, and intrinsic dynamic instabilities  \cite{KORV07,Web10}. The sharp interface limit  as $\vv\to 0$   has been used to describe   the stochastic mean curvature flow   \cite{MR3614757,MR1764701,Yip98}, and to  simulate rare events  \cite{MR3415402,BG18B}.

The  noise term in phase field models is efficient to enhance the stability in simulations, hence has great  interests in the study of  scientific computing \cite{KKL07}.
However, as noted by \cite{Weber10}, conducting numerical analysis near the sharp interface limit remains a significant  challenge due to the complex dynamics that arise at a very small value of $\vv$. A critical issue is deriving a convergent  error bound of numerical approximation with explicit algebraic order of $\vv^{-1}$ for small $\vv\in (0,1).$
Due to the loss of the spectral estimate for the Galerkin approximation equation  (see \eqref{sac} below) in stochastic case, existing  arguments often lead to  error bounds of numerical approximation depending on $\vv^{-1}$ exponentially,
see e.g. \cite{LPS14,BJ16,MR4001780,BCH18,KLL18,FLZ17,CHS21,LQ20,MP17}.
Recently, some progress has been made under strong assumptions:
   under  a smallness condition on the noise \cite{ABNP21}  utilizes  deterministic spectral estimates in \cite{C94} and perturbation techniques to derive an   error bound with algebraic order of $\vv^{-1}$ for an implicit scheme of a stochastic Cahn-Hilliard equation; while   under a non-degenerate condition on the noise  \cite{CS22} derived
an error bound of a  temporal splitting scheme for \eqref{E1} depending on $\vv^{-1}$ at most polynomially. To our best  knowledge, there is no any result on improving the order of $\vv^{-1}$  in the numerical error for  \eqref{E1} without imposing additional smallness  or non-degenerate conditions,  see Remark 2.6 for  detailed explanations.

To derive sharp convergence rate of numerical approximations with explicit algebraic dependence on $\vv^{-1},$   we  introduce a novel approach to estimate the gradient estimate  on the associated Markov semigroup.
A key tool in our method is the asymptotic log-Harnack inequality established in Lemma \ref{L3} below  using  coupling by change of measures.  Although  the asymptotic log-Harnack inequality has been  previously  used to establish the uniqueness of invariant probability measures, asymptotic heat kernel estimates, and the irreducibility of Markov semigroups  \cite{MR2861312,BWY2019}, its application in the numerical analysis of stochastic partial differential equations appears to be novel.  The asymptotic log-Harnack inequality allows us to make sharper error analysis without noise conditions used in \cite{ABNP21, CS22}.

   To realize our techniques, we make  the following assumptions on $A, f$ and $W_t$.

\beg{enumerate} \item[{\bf (A)}]  $(A,\D(A))$ has discrete spectrum  with all eigenvalues $ \{\ll_i>0\}_{i\ge 1}$
 including multiplicities such that $\ll_i\uparrow \infty$ as $i\uparrow\infty$. Next,   the quadric form
 $$\EE(u,v):=\<A^{\ff 1 2}u, A^{\ff 1 2}v\>_\H,\ \ u,v\in \dot\H^1:= \D(A^{\ff 1 2})  $$ is a Dirichlet form,
 where $\<\cdot,\cdot\>_\H$ denotes the inner product in $\H.$
 Moreover,     the following Nash inequality with dimension $d\in (0,\infty)$ holds for   some constant $c>0$:
 \beq\label{NS} \|u\|_{L^2(\mu)}\le c \|u\|_{\dot\H^1}^{\ff{d}{d+2}} \|u\|_{L^1(\mu)}^{\ff 2 {d+2}},\ \ u\in \dot\H^1.\end{equation}
 \item[{\bf (F)}] $f\in C^1(\R)$ satisfies the following conditions   for some constants $m, \kk_1, \kk_2>0$:
 \beq\label{F2}   sf(s) \ge   \kk_2 |s|^{2m+2} -\kk_1,\ \ s\in\R,\end{equation}
\beq\label{F3} f'(s)  \ge  -\kk_1,\ \ \  s \in\R,\end{equation}
\beq\label{F4} |f'(s)|  \le  \kk_1(1+|s|^{2m}),\ \ \ s \in\R.\end{equation}
 \item[{\bf (W)}]  The Wiener process $W_t$ is given by
\beq\label{W} W_t= \sum_{i=1}^\infty \ss{q_i}\, W_t^i e_i,\ \ t\ge 0,\end{equation}
 where  $\{W^i\}_{i\ge 1}$ are   independent one-dimensional Brownian motions      on a probability basis $(\OO, \F, \{\F_t\}_{t\ge 0}, \P),$   $\{q_i \}_{i\ge 1}$  are nonnegative constants satisfying
\beq\label{Q0}  \sum_{i\ge 1} q_i<\infty,\end{equation}
 and  $\{e_i\}_{i\ge 1}$ is  the associated unitary eigenfunctions of $A$ with eigenvalues   $\{\ll_i\}_{i\ge 1}$,  which is an orthonormal basis of $L^2(\mu)$.
 \end{enumerate}

Under the above conditions,   the   existence and uniqueness of the (mild and variational) solution to \eqref{E1}  is standard, see
for instance  \cite[Theorem 6.2.3]{Cer01} and \cite[Lemma 2.2]{CS22}.
By \eqref{Q0} and   the monotonicity of $-A-\ff 1\vv f$,    \cite[Theorem 2.6]{RW08} ensures  It\^o's formula   for $\|u_t\|_{\H }^2$,  and also for
$ \|u_t\|_{\dot \H^\aa}^2 $  if $\aa\in (0,1]$ and
 \beq\label{Q}  \sum_{i=1}^\infty q_i \ll_i^\aa<\infty,\end{equation}
where $\dot\H^\aa := \D(A^{\ff\aa 2})$ and
$$\|u\|_{\dot \H^\aa}^2:=\sum_{i=1}^\infty \ll_i^\aa \<u,e_i\>_\H^2.$$

\begin{exa}\label{ex-f}
A simple example of $(A,f)$ satisfying the above  conditions is that
  $E=D$ is   a bounded regular domain in $\R^d$,  $V\in C^1(D)$ such that $\mu(\d x):=   \e^{V(x)}\d x$ (by shifting   $V$ with a constant which does not change $\nn V$,
   we can always assume that $\e^V$ is a probability density)  is  a probability measure on $D$,   $A= - (\DD+ \nn V \cdot \nn + 1)$ with  Neumann, Dirichlet or mixture boundary condition,
and   $$f (s)= g(s) + c s^{2m+1},\ \ s\in \R,$$
  where $s^{2m+1}:=  |s|^{2m+1}  {\rm sgn}(s),  c>0$ is a constant, and $g$ is a polynomial of degree less than $2m+1$.   Note that in applications, if $A=-(\DD+ \nn V \cdot \nn),$ one can replace the nonlinearity $f$ by $f+1$,
 and still consider  the dominant operator of the form $-(\DD+\nn V\cdot \nn+1).$
  \end{exa}

 To numerically discretize \eqref{E1}, we first consider the following
spectral Galerkin approximation for $N\in\mathbb N$ (the set of positive natural numbers):
 \begin{align}\label{sac}
 \text{d}u^N_t+Au^N_t\text{d}t+\frac 1{\vv}\pi_N f(u^N_t)\text{d}t=\pi_N \text{d}W_t, \ \ u_0^N =\pi_N u_0,
 \end{align} where $\pi_N$ is the spectral projection:
 $$\H\ni x\mapsto \pi_Nx:= \sum_{i=1}^N\<x,e_i\>_\H e_i \in \H_N:=\text{span}\{e_i: 1\le i\le N\}.$$
 Next, we introduce  the fully discrete scheme via a tamed technique (see, e.g., \cite{MR2985171,BCW22}) for \eqref{sac} with a small time stepsize  $\tau\in (0,1]$.
Let
\beg{align}\label{def-kt} &k(t):=\sup\big\{k\in \mathbb N:\ k\tau\le t\big\},\ \ \tau_k:= k \tau,\ \ \ k\in \mathbb N,\ t\ge 0,\\
&\label{def-thetat}
\Theta_{\tau,\sigma}(u):= \ff{-\vv^{-1} f(u)}{1+\tau \|u\|_{\dot \H^{\sigma}}^{2m+1}},\ \
 \Theta_{0}:=\Theta_{0,\sigma}=-\vv^{-1}f,\ \ \si\in (0,\infty).\end{align}
By convention, $\tau_0=0$.
 For any $N\in\N$, the fully discrete tamed Euler scheme is defined by
\begin{align*}
&u_{\tau_{k+1}}^{N,\tau}=\e^{-A\tau} u_{\tau_k}^{N}+ \int_{\tau_k}^{\tau_{k+1}}\e^{-A(\tau_{k+1}-t)}\pi_N \Big(\Theta_{\tau,\sigma}(u_{\tau_k}^{N,\tau})\,\d t+ \d W_t\Big),\\
&  \ \ \ k\in \mathbb N,\ \ u_0^{N,\tau}=u_0^N=\pi_N u_0.
\end{align*}
The continuous interpolation of this discrete scheme is given by the stochastic equation
\beq\label{TF}
 	 \d u_{t}^{N,\tau}= \big\{ \pi_N   \Theta_{\tau,\sigma}(u_{\tau_{k(t)}}^{N,\tau})-Au_{t}^{N,\tau}  \big\}\d t +\pi_N \d W_t,\ \  t\ge 0,\ u_0^{N,\tau}=u_0^N=\pi_N u_0.
 \end{equation}
It is easy to see that this SDE has a unique solution in $\H$, which is continuous in $\dot\H^\beta$ for any $\beta\in (0,\infty)$ since
$\|u\|_\H$ and $\|u\|_{\dot\H^\beta}$ are comparable for $u\in\H_N$.

We aim to   estimate the truncated $L^1$-Wasserstein distance between $u_t$ and $u_t^{N}$ and that between  $u_t$ and $u_t^{N,\tau}$ for large $N$ and small $\tau$:
\begin{align*}
\hat \W_1(u_{t}, u_{t}^{N}):= \sup_{\|\psi\|_{b,1}\le 1} \big|\E\big[\psi(u_{t})-   \psi(u_{t}^{N})\big]\big|,\\
\hat \W_1(u_{t}, u_{t}^{N,\tau}):= \sup_{\|\psi\|_{b,1}\le 1} \big|\E\big[\psi(u_{t})-   \psi(u_{t}^{N,\tau})\big]\big|,
\end{align*}
where  $\psi\in C_b^1(\H)$ and
$$\|\psi\|_{b,1}:= \|\psi\|_\infty+\|\nn\psi \|_\infty$$
with $\|\psi\|_{\infty}=\sup\limits_{x\in \H} |\psi(x)|$ and $\|\nn \psi \|_{\infty}=\sup\limits_{x\in \H}\|\nabla \psi(x)\|_{\H}.$

\
The remainder of the paper is organized as follows.
In Section \ref{sec-2},  we state  the main results of the paper, which provide explicit decay rates of  $\hat \W_1(u_{t}, u_{t}^{N})$ and $\hat \W_1(u_{t}, u_{t}^{N,\tau})$  as $N\to\infty$ and $\tau\to 0$,  and with explicit algebraic order of $\vv^{-1}$ for small $\vv\in (0,1)$. These convergence  rates in $N\to\infty$ and
$\tau\to 0$ can be sharp as shown in Remark \ref{Remark 2.1.}(ii) and  Remark \ref{Remark 2.2}. Noting  that the existing upper bound of numerical  weak error for \eqref{E1} derived in the literature  is exponential in $\vv^{-1}$ (see \cite{BG18B,CH18,CHS21,Bre22,MR4756577} and references therein), the contribution of the present work is reducing
the order from exponential to algebraic.
In Section \ref{sec-3},  we present some moment estimates  on $u_t, u_t^N$ and $u_t^{N,\tau}$ with algebraic order of $\vv^{-1}$.
To deal with the possibly degenerate noise, in Section  \ref{sec-4}  we make use of the essential elliptic condition to show the asymptotic log-Harnack inequalities
for both Eq. \eqref{sac} and its spectral Galerkin approximation, which implies  the regularity and  asymptotic strong Feller property  of the associated Markov semigroup.
Finally, with above preparations, Sections \ref{sec-5}-\ref{sec-6} provide complete proofs of the main results.

\section{Main results}
\label{sec-2}

Let  $\kk_1$ be in \eqref{F2}. For any $\vv\in (0,1)$ and $N\in\mathbb N$,   define
\beq\label{AN} \beg{split}
 &N_{\vv}:=\inf\Big\{N\in \mathbb N: \ll_{N+1} \ge \ff{1+\kk_1}\vv\Big\},\\
& \gg(\vv):= \sup_{1\le i\le N_{\vv}} q_i^{-\ff 1 2},\ \ \ \
 \dd_N:= \sum_{i=N +1}^\infty q_i\ll_i^{-1}.\end{split}\end{equation}
 Note that  \eqref{Q0} implies  $\dd_N\le \ll_{N+1}^{-1}$ which goes to $0$ as $N\to\infty$, and when $\eqref{Q}$ holds we have better estimate $\dd_N\le \ll_{N+1}^{-1-\aa}$. Moreover, to ensure the finiteness of $\gg(\vv)$ we only need $q_i>0$ for $i\le N_\vv$, i.e. the Wiener process $W_t$ in \eqref{W} may be degenerate with $q_i=0$ for $i>N_\vv$.

  For any $x\in \H$, let $u_t(x)$ solve \eqref{E1} with $u_0=x$, and let $u_t^N(x)$ solve \eqref{sac} for $u_0^N=\pi_Nx$. The  following result provides some convergence rates for the spectral Galerkin approximation,  which can be sharp as shown in Remark \ref{Remark 2.1.}(ii) below.

 \beg{thm}\label{T1}  Assume {\bf (A)}, {\bf (F)} and {\bf (W)} with $ \aa_{m,d}:= \ff{md}{2m+1}\le 1$,  let $\eqref{Q}$ hold for some  $\aa\in [\aa_{m,d},1]$   and  let  $\gamma(\vv)<\infty$ in $\eqref{AN}$.
Then
  the following assertions hold.
 \beg{enumerate}\item[$(1)$]    There exists  a constant $c>0$ such that for any $x\in \dot \H^{\aa_{m,d}},t>0,$   and $N\in\N ,$
 \beq\label{BJ1}\beg{split}    &\hat\W_{1}\big(u_t(x),  u_t^N(x)\big) \\
 &\le c\gg(\vv) \vv^{-\ff{md+5}2} (1+t)\big(1+\|x\|_{\dot\H^{\aa_{m,d}}}^{2m+1}\big)\big(\|x-\pi_Nx\|_\H + \dd_N^{\ff 1 4}+\ll_{N+1}^{-\ff 1 2}\big).
  \end{split}\end{equation}
 \item[$(2)$]  If $d\in (0,2)$, then for any $\beta \in (\frac d2,1],$ there exists  a constant $c>0$ such that
  \beq\label{BJ2}\beg{split}   & \hat\W_{1}\big(u_t(x),  u_t^N(x)\big) \\
  &\le c\gg(\vv)\vv^{-\beta m-\ff{md+5}2}(1+t) \big(1+\|x\|_{\dot\H^{\beta}}^{4m+1}\big)\big(\|x-\pi_Nx\|_\H+\dd_N^{\ff 1 2}+  \ll_{N+1}^{-1}\big).
  \end{split}\end{equation}
 \item[$(3)$]   If $d\in [2,2+m^{-1}]\cap [2, 4)$, then for any   $\bb\in (\ff d 2,2)\cap [\alpha,1+\alpha]$, there exist a constant $c>0$ such that
 \beq\label{BJ3}\beg{split}    &\hat\W_{1}\big(u_t(x),  u_t^N(x)\big)\\
  &\le c \gg(\vv)\vv^{-m(md+2)-\ff{md+5}2}(1+t) \big(1+\|x\|_{\dot\H^\bb}^{(2m+1)^2}\big)\big(\|x-\pi_Nx\|_\H+\dd_N^{\ff 1 2}+  \ll_{N+1}^{-1}\big).
  \end{split}\end{equation}
  \end{enumerate}
 \end{thm}

\begin{rem}\label{Remark 2.1.}
\begin{enumerate}
\item[(i)] Noting that  \eqref{NS} implies
$\ll_N\ge c_1N^{\ff 2 d} $ for some constant $c_1>0$, we have that in \eqref{AN} $N_\vv\le c_2 \vv^{-\ff d 2}$ for some constant $c_2>0$. So, for $q_i\sim i^{-m_1}$ for some $m_1>1$, we find a constant $c_3>0$ such that
$$\gg(\vv)\le c_3 \vv^{-\ff{m_1d}4},\ \ \ \dd_N\le c_3 N^{1-m_1-\ff 2 d},$$
so that   \eqref{BJ1}-\eqref{BJ3} present   algebraic convergence rate  $N^{-a_1}$ and multiplication $\vv^{-a_2}$ with explicit $a_1,a_2>0.$  Comparing with
\eqref{BJ1}, the other estimates  \eqref{BJ2}-\eqref{BJ3}  provide faster convergence order  in $N^{-1}$ but with larger multiplication in $\vv^{-1}$.
 \item[(ii)] The error term $\|x-\pi_Nx\|_\H+ \dd_N^{\ff 1 2}$ is the exact convergence rate as $N\to\infty$  for $f=0$. So, the convergence rate in \eqref{BJ2}  and \eqref{BJ3} are sharp when $\ll_{N+1}^{-1}\le c \dd_N^{\ff 1 2}$ for some constant $c>0,$ it is the case if $\ll_N \sim N^{\ff 2 d}$ and $q_N\sim N^{-1-\theta}$ for some $\theta\in (0,\ff 2 d]$.
\item[(iii)] By imposing additional regularity on $f$, specifically assuming that  $f\in C^2(\R)$, one  can enhance the spatial weak convergence order w.r.t. $N^{-1}$ via the corresponding Kolmogorov equation (see, e.g., \cite{MR4756577,CH18,CJK19,MR4797879} and references therein). For instance, considering  the polynomial in Example \ref{ex-f} with an integer degree,   it can be shown that the second derivative of the solution of the corresponding Kolmogorov equation is finite.
Thanks to Lemma \ref{L4} and the Kolmogorov equation, the convergence rate $\mathcal O(\|x-\pi_Nx\|_\H+  \ll_{N+1}^{-1}+\dd_N^{\ff 1 2})$ in Theorem \ref{T1} (2)-(3) can  be improved as follows: for  any $\alpha_1\in [0,2)$, there exists a constant $c(x,\vv^{-1},t,\psi)>0$ such that 
any $x\in \dot \H^{\beta}$ with $\beta\in (\frac d2,2)$,  $t>0$ and $\psi\in C_b^2( \H)$,
\begin{align*}
&|\E [{ \psi}(u_t(x))]-\E[{\psi}(u_t^N(\pi_N x))]|\\
&\le c(x,\vv^{-1},t,\psi) (1+t^{-\frac {\alpha_1} 2}) \Big(\lambda_N^{-\frac {\alpha_1} 2-\frac {\beta} 2}+\delta_N^{\frac 12} +\|(1-\pi_N)x\|_{{\dot \H}^{-\alpha_1}}\Big).
\end{align*}
Here $c(x,\vv^{-1},t,\psi)$ is a positive polynomial of $\|x\|_{\dot \H^{\beta\vee 1}}$,  $\vv^{-1}$, $\|\nabla \psi\|_{\infty}+\|\psi\|_{\infty}$. The space $\dot \H^{-\alpha_1}$ is the dual space of   $\dot \H^{\alpha_1}$. Note that the degree of  $\vv^{-1}$ in $c(x,\vv^{-1},t,\psi)$ is higher than the degrees in \eqref{BJ2}-\eqref{BJ3}, see Appendix \ref{appendix} for details.
\end{enumerate}

\end{rem}

We now present the main convergence result of the tamed fully discrete scheme with time stepsize $\tau\in (0,1)$ and constant $\si$ in \eqref{def-thetat} satisfying
\beq\label{SGM}\si\in \beg{cases}[\aa_{m,d},1], &\text{if}\ \aa_{m,d}<1,\\
(1,2), &\text{if}\ \aa_{m,d}=1.\end{cases}\end{equation} 
Then, when $\aa_{m,d}<1$ we have  
$$\hat q:= \ff{d}{(2d)\land (d+2(1-\si)^+)- 2m (d-2\si)^+}<\ff d{(d-2\si)^+}.$$
 For fixed  constants
 \beq\label{QD}\beta \in (\aa_{m,d}, 2),\ \ \   q_{d,\si}\in \Big[\hat q,\ \frac {d}{(d-2\sigma)^+}\Big)\ \text{when}\  \aa_{m,d}<1,\end{equation}
    let
\beq\label{DD}\beg{split}
&\gg_1:= \ff 1 2 (2m+1) (\beta-\aa_{m,d}),\\
&\gg_2:=   \max\bigg\{\ff{4q_{d,\sigma}}{d+(2\sigma-d)q_{d,\sigma}},\ \ff 2 {\sigma}\bigg\}\ \text{when}\ \aa_{m,d}<1,\\
&\gg_3:=  \ff{2}{\sigma-1}  \ \text{when}\ \aa_{m,d}=1,\\
   & a_{t,x}:= (1+t)(1+\|x\|_{\dot\H^\beta}),\ \ \ t\ge 0, x\in\dot\H^\beta.\end{split}\end{equation}

 We have the following result for the convergence of the tamed Euler scheme, where the convergence rate can be sharp as shown in Remark \ref{Remark 2.2}.

\beg{thm}\label{T2}
  Assume {\bf (A)}, {\bf (F)} and {\bf (W)} with $ \aa_{m,d}:= \ff{md}{2m+1}\le 1$,   let $\eqref{Q}$ hold for   some $\aa\in [\aa_{m,d},1],$   let $\si$ satisfy $\eqref{SGM}$, and let
   $\gg_1,\gg_2,\gg_3$ and $a_{t,x}$ be defined in $\eqref{DD}$ for constants $\bb$ and $q_{d,\si}$ satisfying $\eqref{QD}$.
  Then
  the following assertions hold for all $x\in\dot\H^\beta, t\ge 0, \vv\in (0,1]$, $N\in \mathbb N$ and small $\tau>0$ satisfying
  \beq\label{TT}  \tau \le  \beg{cases}\vv^{ [\ff 9 2+ 12(2m+1)(m+1)]\gg_2}   (1+\|{x}\|_{\dot\H^\sigma})^{1-6(2m+1)^2\gg_2}, &\text{if} \; \aa_{m,d}<1, \\
 \vv^{[\frac 92+ 3(2m+1)^2(2m+3)+6(2m+1)]\gamma_3-\ff {2m+1}2}(1+\|{x}\|_{\dot\H^\sigma})^{-(2m+1)[6\gamma_{3}(2m+1)^2 -1]}, &\text{if} \; \aa_{m,d}=1.\end{cases}
 \end{equation}
\beg{enumerate} \item[$(1)$] When $\alpha_{m,d}<1$ and $\bb\in [\si,1],$ there exists a  constant $c>0$ such  that
      \beq\label{A} \hat\W_1(u_t^N(x), u_t^{N,\tau}(x))
\le c  \tau^{\gg_1\wedge \frac 12}  (1+t) \gg(\vv) \vv^{-\frac 72-4m} a_{t,x}^{4m+2}.
\end{equation}
  \item[$(2)$] When $\alpha_{m,d}<1$ and $\beta\in (\sigma,1+\alpha]$, there exists a constant $c>0$ such that
     \beq\label{B} \beg{split}  &\hat\W_1(u_t^N(x), u_t^{N,\tau}(x))  \le  c  \tau^{\gg_1\land \frac 12} (1+t)  \gg(\vv)\vv^{-\frac 12-(2m+2)^2} a_{t,x}^{ (2m+1)(2m+2)}.\end{split} \end{equation}
  \item[$(3)$] When $\aa=\aa_{m,d}=1$ and   $\beta \in [\sigma , 2)$,
  there exists a constant $c>0$ such that
     \beq\label{B'} \beg{split}  &\hat\W_1(u_t^N(x), u_t^{N,\tau}(x))  \\
     &\le  c  \tau^{\gg_1\land \frac 12} (1+t) \gg(\vv)\vv^{-(m+1)(2m+1)(2m+3)-2m-\frac 32}a_{t,x}^{2(m+1)(2m+1)^2}.\end{split} \end{equation}
\end{enumerate}
  \end{thm}
\begin{rem}\label{Remark 2.2}
\begin{enumerate}
 We may take $\beta \in [\ff{md+1}{2m+1},2)$ such that  $\gg_1\ge \ff 1 2$.
As a consequence,  Theorem \ref{T2} provides  the sharp convergence rate  $\tau^{\ff 1 2}$  in the tamed Euler scheme    for SDEs.
\end{enumerate}
\end{rem}

Combining Theorem \ref{T1}(2) with  \eqref{A} in Theorem \ref{T2}(1),    respectively Theorem \ref{T1}(3) with  \eqref{B} in Theorem \ref{T2}(2)-(3),   and noting that when
$\aa\ge 1$ the condition \eqref{Q} implies $\dd_N\le  \ll_{N+1}^{-2} $ for large $N$,
we derive the following overall error estimate on $\hat\W_1(u_t,u_t^{N,\tau})$ for $\vv\in (0,1),$ large $N\ge 1$ and small $\tau>0$.

\begin{cor}\label{cor-main}  In the situation of Theorem $\ref{T2}$, the following assertions hold for all $x\in\dot\H^\beta, t\ge 0, \vv\in (0,1]$, $N\in \mathbb N$ and small $\tau>0$ satisfying $\eqref{TT}.$
 \beg{enumerate} \item[$(1)$] When $\aa_{m,d}<1, d\in (0,2),$ and $\beta\in (\frac d2,1]\cap [\sigma, 1],$  there exists a constant $c>0$ such that
  \beg{align*}   & \hat\W_{1}(u_t(x), u_t^{N,\tau}(x))
   \le   c\gg(\vv)(1+t)\\
   &\qquad\quad\times \Big[\tau^{\frac 12\land \gg_1} \vv^{-\ff 7 2-4m} a_{t,x}^{4m+2}
   +\vv^{-\beta m-\ff{md+5}2}\big(1+\|x\|_{\dot\H^{\beta}}^{4m+1}\big) \big(\|x-\pi_Nx\|_\H +\dd_N^{\ff 1 2}+\ll_{N+1}^{-1} \big)
  \Big].
\end{align*}
\item[$(2)$] When $\aa_{m,d}<1, d\in [2, 2+m^{-1}]\cap [2,4),$ and $\beta \in (\ff d 2 ,2)\cap [\sigma,1+\aa]$,  there exists a constant $c>0$ such that
   \beg{align*}   & \hat\W_{1}(u_t(x), u_t^{N,\tau}(x))  \le
     c  \gg(\vv)(1+t)
      \Big[ \tau^{\gg_1\land \frac 12} \vv^{-\frac 12-(2m+2)^2} a_{t,x}^{ (2m+1)(2m+2)}  \\
&\qquad \qquad\qquad\qquad+ \vv^{-m(md+2)-\ff{md+5}2} \big(1+\|x\|_{\dot\H^\beta}^{2m(2m+1)}\big)\big(\|x-\pi_Nx\|_\H+ \dd_N^{\ff 1 2}+  \ll_{N+1}^{-1}\big) \Big].
\end{align*}
\item[$(3)$] When $\aa=\aa_{m,d}=1$ and $\beta\in (\frac d 2 , 2)\cap [\sigma,2),$
 there exists a constant $c>0$ such that
   \beg{align*}   & \hat\W_{1}(u_t(x), u_t^{N,\tau}(x))
   \le  c  {\gg(\vv)}(1+t)\Big[ \tau^{\gg_1\land \frac 12} \vv^{-(m+1)(2m+1)(2m+3)-2m-\frac 32}a_{t,x}^{2(m+1)(2m+1)^2}  \\
&\qquad \qquad \qquad\qquad+ \vv^{-m(md+2)-\ff{md+5}2} \big(1+\|x\|_{\dot\H^\beta}^{2m(2m+1)}\big)\big(\|x-\pi_Nx\|_\H+ \dd_N^{\ff 1 2}+  \ll_{N+1}^{-1}\big) \Big].
\end{align*}
\end{enumerate}
\end{cor}

\begin{rem}\label{Remark 2.6.} Let us compare our result with those derived in \cite{ABNP21} and \cite{CS22}
with polynomial dependence on $\vv^{-1}$.

In \cite{ABNP21}, when $W_t$ is replaced by $\vv^{\gamma}gB_t$ for a constant $\gg>0$, a one-dimensional Brownian motion $B_t$ and a smooth function $g$,   the authors use  deterministic spectral estimates in \cite{C94} to show the polynomial dependence on $\vv^{-1}$ of the strong error for the drift implicit Euler scheme of a stochastic Cahn--Hilliard equation. The techniques in \cite{ABNP21} make use of the properties of the strong solution (see, e.g. \cite[Theorem 3.8, Theorem 4.3]{ABNP21}) and thus require more  regularity condition on the initial data $(\text{like}\; u_0\in \dot \H^{3})$ and the noise $($like $W$ is smooth in space$)$.
In this paper, without any smallness of noise, and for mild conditions on the initial data,
 we show the  polynomial dependence on $\vv^{-1}$ of the numerical error under a truncated Wasserstein metric for the proposed schemes. Note that the week error under the truncated Wasserstein metric is bounded by the strong error under the $L^p(\Omega;\H)$ norm. Our approach also has a great potential to study the numerical weak error bound of stochastic Cahn--Hilliard equation, and this will be reported in another paper.

In \cite{CS22}, under a non-degenerate condition on the noise $($i.e., $q_i>0$ for all $i\in {\mathbb N})$, for the stochastic Allen--Cahn equation, the authors make use of the strong Feller property and exponential ergodicity to show that the weak error with certain test function of a semi-discrete splitting scheme depends on $\vv^{-1}$ polynomially. The result in \cite{CS22} (Theorem 1.1)
could cover the space-time white noise case $(d=1, q_i=1, \alpha< -\frac 12\; \text{in}\; \eqref{Q})$, while our main focus is on the cases that  the noise may be degenerate (i.e., \eqref{Q0}). Indeed, our result does not reply on the exponential ergodicity or the strong Feller property, and can deal with the degenerate noise as soon as $\gamma(\vv)<\infty$ in $\eqref{AN}$.  By changing the tamed scheme into other time discretization, such as the temporal semidiscrete splitting scheme \cite{CS22}, one may improve the restriction on the time stepsize \eqref{TT} for the numerical scheme.
\end{rem}

 \section{Moment estimates on $u_t, u_t^N$ and $u_t^{N,\tau}$}
\label{sec-3}
 In this section, we present some estimates on  the solutions    to   the original problem \eqref{E1}, the  spatial Galerkin approximation \eqref{sac}, and the numerical scheme \eqref{TF},
  which will be used in the proofs of Theorems \ref{T1} and \ref{T2}.

We first observe that by \cite[Corollary 3.3.4]{Wbook}, \eqref{NS} is equivalent to that for some constant $c>0$
 \beq\label{FS} \|u\|_{2+\ff 4 d}\le c\|u\|_{\dot\H^1}^{\ff d{d+2}} \|u\|_{L^2(\mu)}^{\ff 2{d+2}},\ \ u\in \D(\EE),\end{equation}
 and when $d>2$, it is also equivalent to the Sobolev inequality
 $$ \|u\|_{L^{\ff{2d}{d-2}}(\mu)} \le c \|u\|_{\dot\H^1},\ \ u\in \D(\EE) $$ for some (different) constant $c>0$.
 Moreover,   by \cite[Theorem 1.1]{BM}, for any $\aa\in (0,1],$ \eqref{NS} implies
the same type inequality for $(\|\cdot\|_{\dot\H^\aa}, d/\aa)$ replacing $(\|\cdot\|_{\dot\H^1},d)$, and hence
\beq\label{S}  \|u\|_{L^{\ff{2d}{d-2\aa}}(\mu)} \le c(\aa) \|u\|_{\dot\H^\aa},\ \      \aa\in (0,1]\cap \Big(0, \ff d 2\Big),\end{equation}
where  $c(\aa)\in (0,\infty)$ is a constant depending on $\aa$.

  For  $N\in\bar\N:=  \N \cup \{\infty\}$. By \eqref{Q},
\beq\label{Z0} Z_t^N  :=\int_0^t \pi_N \e^{- A  (t-s) } \d W_s=\sum_{i=1}^N q_i^{\ff 1 2} e_i \int_0^t \e^{- \ll_i  (t-s)}\d W_i(s),\ \ t\ge 0 \end{equation}  is a continuous Gaussian process on $\H$.
By \eqref{sac}, we have
\beq\label{Z1} u_t^N= \e^{-At} u_0^N-\ff 1\vv\int_0^t \pi_N \e^{A(t-s)}  f(u_s^N)\d s+ Z_t^N,\ \ t\ge 0,\end{equation}
where $u_t=u_t^N$ for  $N=\infty$. Moreover, for any $N\in\mathbb N$, \eqref{TF} implies
\beq\label{Z1'} u_t^{N,\tau}= \e^{-At} u_0^N+  \int_0^t \pi_N \e^{A(t-s)}  \Theta_{\tau,\sigma} (u_{\tau_{k(s)}}^{N,\tau}) \d s + Z_t^N,\ \ t\ge 0.\end{equation}

\subsection{Estimates on $u_t$ and $u_t^N$}

We first estimate the Gaussian process $Z_t^N$.

\beg{lem}\label{L1}  Assume {\bf (W)} and let $\eqref{Q}$ hold for   some $\aa\in [0,1].$ For any constant $p\ge 2$, there exists a constant $c >0$ such that for all $N\in\bar\N$ and $\beta \in [\alpha,1+\alpha]$,
\beg{align} \label{ES1} &\sup_{t\ge 0} \|Z_t^N\|_{L^p(\OO;\dot\H^\beta)}\le c \Big(\sum_{i=1}^N \ll_i^{\beta-1}q_i\Big)^{\ff 1 2},\\
 \label{ES2}  &\|Z_t^N-Z_s^N\|_{L^p(\OO;\H)}\le c \Big(\sum_{i=1}^N q_i\Big)^{\ff 1 2} \ss{t-s},\ \ t\ge s\ge 0.\end{align}
\end{lem}

\beg{proof} For any $t\ge s\ge 0$, let
$$  \xi_i^{s,t}:= \int_s^t \e^{- \ll_i (t-r)}\d W_i(r).$$
Then   $\xi_i^{s,t}$ is a centered normal random variable  with variance
$$  {\text{Var}}(\xi_i^{s,t})^2:=  \int_s^t \e^{-2\ll_i (t-r)}\d r.$$
So, for any $p\ge 2$ there exists a constant $c_1  >0$ such that
\beq\label{V}
\E\big[|\xi_i^{s,t}|^p\big]  \le c_1  \bigg(\int_s^t \e^{-2\ll_i (t-r)}\d r\bigg)^{\ff p 2} \le c_1 \Big((t-s)\land \ff 1 {\ll_i}\Big)^{\ff   p 2}. \end{equation}
Moreover, it is easy to see that
\beg{align*}  Z_t^N-Z_s^N &=\sum_{i=1}^N e_i \bigg[ (\e^{-\ll_i  (t-s)}-1)  \<Z_s^N, e_i\> +q_i^{\ff 1 2} \int_s^t\e^{-\ll_i (t-r)} \d W_i(r)\bigg] \\
&= \sum_{i=1}^N e_i q_i^{\ff 1 2}  \Big[(\e^{-\ll_i (t-s)}-1)\xi_i^{0,s} +\xi_i^{s,t}\Big].\end{align*}
Combining this with H\"older's inequality,   for any $\bb_1\ge 0$ we have
  \beg{align*}& \E \big[\|Z_t^N-Z_s^N\|_{\dot\H^\beta}^p\big]\le 2 \E\Big[\sum_{i=1}^N q_i  \ll_i ^\beta \big((\e^{-\ll_i (t-s)}-1)^2|\xi_i^{0,s}|^2 + |\xi_i^{s,t}  |^2\big) \Big]^{\ff p 2} \\
&\le 2 \Big(\sum_{i=1}^N \ll_i ^{\beta -\bb_1}  q_i\Big)^{\ff{p-2}2}
  \sum_{i=1}^N \ll_i ^{\beta+ \ff{\bb_1(p-2)}2} q_i  \Big(|\e^{-\ll_i (t-s)}-1|^p \E[ |\xi_i^{0,s} |^p]+ \E [|\xi_i^{s,t}  |^p]\Big).\end{align*}
By taking  $\bb_1=1$ and $s=0$, we deduce \eqref{ES1} from this and  \eqref{V}, while    taking $\bb=\bb_1=0 $  we obtain
\beg{align*}& \E \big[\|Z_t^N-Z_s^N\|_{\H}^p\big]\\
& \le 2 c_1  \Big(\sum_{i=1}^Nq_i \Big)^{\ff {p-2}2} \sum_{i=1}^N q_i\Big(\big[\ll_i (t-s)\big]^{\ff p 2} \ll_i ^{-\ff p 2} +(t-s)^{\ff p 2}\Big)\\
&\le 4 c_1   \Big(\sum_{i=1}^Nq_i\Big)^{\ff {p}2}(t-s)^{\ff p 2}.\end{align*}
This and \eqref{Q0} imply  \eqref{ES2}   for some constant $c>0$.
\end{proof}

  Next, we estimate  moments of $\|u_t^N\|_{ \dot\H^\beta}$  for $\beta\in [0,2).$

\beg{lem}\label{L2}  Assume   {\bf (A)}, {\bf (F)} and {\bf (W)}.
\beg{enumerate} \item[$(1)$] For any $p\ge 2$ and $\kk>0$, there exists a constant $c>0$ such that
\beq\label{*2} \sup_{ N\in\bar\N}  \|u_t^N\|_{L^p(\OO; \H)} \le   \|u_0\|_\H \e^{-\ff \kk {\vv}t} +c,\ \ t>0.\end{equation}
\item[$(2)$]
Let $\eqref{Q}$ hold for some
  $\alpha\in [0,1].$ Then for any constant $p\ge 2$, there exists a constant $c>0$ such that
\beq\label{*2'}\sup_{t\ge 0,  N\in\bar\N} \|u^N_t\|_{L^p(\Omega; \dot\H^{\alpha})}\le \|u_0\|_{ \dot\H^{\alpha}}+c\vv^{-\frac \alpha 2}(1+\|u_0\|_{\H}).\end{equation}

\item[$(3)$] Let  $\aa_{m,d}:=\ff{md}{2m+1}\le 1$ and $\eqref{Q}$ hold for  some $\aa\in [ \aa_{m,d},1].$ Then for any  $\bb \in [\aa,1+\aa]\cap (1,2)$,   there exists a constant  $c>0$    such that
\begin{align}\label{*2''}
&\sup_{t\ge 0, N\in\bar\N}\|u^N_t\|_{L^p(\Omega;\dot\H^\bb)}\le \|u_0\|_{\dot\H^\bb} +c \vv^{-1}\|u_0\|_{\dot\H^{\aa_{m,d}}}^{2m+1}  +c    \vv^{- \ff{md+2}2}(1+\|u_0\|_\H^{2m+1}).
\end{align}

\end{enumerate} \end{lem}

\beg{proof}
  (1)  By \eqref{Q0},  \eqref{F2},     It\^o's formula and
  $$\|\cdot \|_\H=\|\cdot\|_{L^2(\mu)} \le \|\cdot\|_{L^{2m+2}(\mu)},$$  for any $q\in [1,\infty)$ we find   a constant $c_1  >0$ and a martingale $M_t$
  such that
\beg{align*} \d \|u_t^N\|_\H^{2q}&\le \d M_t +
  2q\|u_t^N\|_\H^{2q-2}\Big(\ff {\kk_1 } \vv + c_1 -\ff{\kk_2} \vv \|u_t^N\|_{L^{2m+2}(\mu)}^{2m+2}\Big)\d t\\
  &\le  \d M_t +
  2q\|u_t^N\|_\H^{2q-2}\Big(\ff {\kk_1 } \vv
  + c_1 -\ff{\kk_2} \vv \|u_t^N\|_{\H}^{2m+2}\Big)\d t.\end{align*}
  Indeed, $\d M_t =2q\|u_t^N\|_\H^{2q-2}\<u_t^N,\d W_t\>_\H.$
  By Young's inequality, for any constant $\kk>0$, we find a constant $c_2>0$ such that
  $$\frac {\kk_2}{\vv} \|u_t^N\|_{\H}^{2m+2q}\ge \frac {\kk}{\vv} \|u_t^N\|_\H^{2q}+(\frac { \kk_1}{\vv}+c_1 )\|u_t^N\|_\H^{2q-2} -\frac {c_2}{\vv},$$
  so that
  $$\d \|u_t^N\|_\H^{2q}\le 2q\|u_t^N\|_\H^{2q-2}\<u_t^N,\d W_t\>_\H +\Big(\ff {2qc_2}\vv -\ff{2q\kk}\vv \|u_t^N\|_\H^{2q} \Big)\d t.$$
   By taking expectation and applying Gronwall's inequality, we obtain
   \beg{align*}   \E[\|u_t^N\|_\H^{2q}]\le \|u_0\|_\H^{2q}\e^{-\ff{2q\kk}\vv t}  + \ff{2qc_2}{\vv}  \int_0^t \e^{-\ff{2q\kk}{\vv}s} \d s\le  \|u_0\|_\H^{2q}\e^{-\ff{2q\kk}\vv t} +\ff {c_2}{\kappa}.\end{align*}
So,  \eqref{*2}  holds for $p=2q$ and $c=(c_2\kappa^{-1})^{\ff 1 {2q}}.$

(2)  Let $u_0\in \D(A^{\ff \aa 2}) $  and
 \eqref{Q} hold. Since $\aa\in (0,1]$ and
$$ \<A^{\ff 1 2}  u, A^{\ff 1 2}  v\>_\H,\ \ u,v\in \D(A^{\ff 1 2})$$ is a Dirichlet form,
$$ \<A^{\ff \aa 2}  u, A^{\ff \aa 2}  v\>_\H,\ \ u,v\in \D(A^{\ff \aa 2}) $$
is a Dirichlet form as well.
So,  \eqref{F3} implies
\beq\label{MY0} \<A^{\ff \aa 2}  u, A^{\ff \aa 2}  f(u)\>_\H \ge -\kk_1 \<A^{\ff \aa 2}  u, A^{\ff \aa 2} u\>_\H= -\kk_1 \|u\|_{\dot\H^\aa}^2,\ \ \aa\in (0,1].\end{equation}
Combining this with  It\^o's formula and \eqref{Q},  for any $q\ge 1$,  we find a constant $c_0>0$
 and a martingale $M_t$ such that
\beq\label{L01} \d \|u_t^N\|_{\dot\H^\aa}^{2q} \le   \d M_t+
   2 q \|u_t^N\|_{ \dot\H^{\alpha}}^{2q-2}\Big( \ff{\kk_1}\vv \|u_t^N\|_{\dot\H^\aa}^2   + c_0 - \|A^{\ff{\aa+1}2} u_t^N\|^2_\H \Big)\d t. \end{equation}
By the H\"older and Young inequalities,   we find some constant $c_1>0$ such that
\beg{align*} & \|u_t^N\|_{\dot\H^\aa}^2 =\sum_{i=1}^\infty \ll_i^\aa \<u_t^N,e_i\>_\H^2\le \bigg(\sum_{i=1}^\infty \ll_i^{1+\aa} \<u_t^N,e_i\>_\H^2\bigg)^{\ff {\alpha}{1+\alpha}}
\bigg(\sum_{i=1}^\infty   \<u_t^N,e_i\>_\H^2\bigg)^{\ff 1{1+\alpha}}\\
& \le  \|A^\ff{\aa+1}2 u_t^N\|_\H^{\frac {2\alpha}{\alpha+1}} \|u_t^N\|_\H^{\frac 2{1+\alpha}}  \le \ff{\vv}{2\kk_1} \|A^{\ff{\aa+1}2} u_t^N\|_\H^2+{c_1} \vv^{-  \alpha} \|u_t^N\|_\H^2.\end{align*}
Thus,
\beq\label{MY}  \|A^{\ff{\aa+1}2} u_t^N\|_\H^2\ge \ff{2\kk_1}\vv \|u_t^N\|_{\dot\H^\aa}^2- 2\kk_1c_1\vv^{-1-\alpha} \|u_t^N\|_\H^2.\end{equation}
Combining this with \eqref{L01} and    Young's inequality
$$\ff{\kk_1}{2} r^{2q}\ge a r^{2q-2} - c(q,\kk_1)  a^q,\ \ a\ge 0, r\ge 0$$
for some constant $c(q,\kk_1)>0$,
  we find a constant $c_2>0$ such that
  \beg{align*} &\d \|u_t^N\|_{\dot\H^\aa}^{2q}-\d M_t\\
  &\le  2q \|u_t^N\|_{\dot\H^\aa}^{2q-2}\Big( 2c_1\kappa_1\vv^{- 1-\alpha}   \|u_t^N\|_\H^2  +c_0   -\frac {\kk_1}{\vv} \|u_t^N\|_{\dot\H^\aa}^{2} \Big) \d t\\
  &\le 2q\Big(  c_2+  c_2 \vv^{-   1-\alpha q} \|u_t^N\|_\H^{2q} - \ff{\kk_1}{2\vv} \|u_t^N\|_{\dot\H^\aa}^{2q}\Big)\d t.\end{align*}
Combining this with \eqref{*2} and Gronwall's inequality, we find a constant $c_3>0$ such that
\beg{align*}     \E\big[\|u_t^N\|_{\dot\H^\aa}^{2q}\big]
&\le \|u_0\|_{\dot\H^\aa}^{2q} \e^{-\ff{q\kk_1}{2\vv} t} +  2q\int_0^t \e^{-\ff{q\kk_1}{2\vv} (t-s)} \Big(  c_2 +  c_2 \vv^{-   1-\alpha q } \E \|u_s\|_\H^{2q} \Big)\d s\\
& \le  \|u_0\|_{\dot\H^\aa}^{2q} \e^{-\ff{q\kk_1}{2\vv} t} +    c_3 \vv^{-  \alpha q } \big(1+  \|u_0\|_\H^{2q} \e^{-\ff{q\kk_1}{2\vv} t} \big).\end{align*}
This implies \eqref{*2'}.

(3)   Let $\bb\in [\alpha,1+\alpha]\cap (1,2)$. By \eqref{ES1} and \eqref{Q}, we have
\beq\label{NM0}\sup_{t\ge 0,N\in\bar\N} \|Z_t^N\|_{L^p(\OO; \dot\H^\bb)}<\infty.\end{equation}
Since $\aa_{m,d}:=\ff{md}{1+2m}\in (0,1]$,   we have  $d\le 2+\ff 1 m $ and
\beq\label{NM0'}
  4 m+2= \ff {2d}{d-2\aa_{m,d}}=\ff{2d/\aa_{m,d}}{d/\aa_{m,d}-2}.\end{equation}
 By \eqref{S} for $\aa_{m,d}\le 1$, there exists a constant $c_1>0$ such that
 \beq\label{PB}
 \|u\|_{L^{4m+2}(\mu)} \le c_1 \|u\|_{\dot\H^{\aa_{m,d}}}.
  \end{equation}
This together with \eqref{*2'}, which holds for the present $\aa_{m,d}=\ff{md}{1+2m}$ due to \eqref{Q}, we obtain
    \beq\label{NM} \|u_t^N\|_{L^p(\Omega;L^{4m+2}(\mu))} \le c_1 \|u_0\|_{\dot\H^{\aa_{m,d}}}+ c_1c\vv^{-\ff {\aa_{m,d}} 2}(1+\|u_0\|_\H).\end{equation}
Noting that \eqref{F4} implies
\beq\label{F1} |f(s)|\le \kk_1 (1+|s|^{2m+1}),\ \ s\in\R \end{equation}
 for  $\kk_1>0$, and
 \beq\label{BN0} \|\e^{-Ar} u\|_{\dot\H^\bb}\le \sup_{\ll\ge \ll_1} \e^{-\ll r}\ll^{\ff\bb 2}
\|u\|_\H  \le c(\bb)r^{-\ff\bb 2} \e^{-\ff{\ll_1}2 r}\|u\|_\H, \ \ r>0\end{equation} for some constant $c(\bb)>0$,
  we find a constant $c_2>0$ such that
$$\|\e^{-Ar} \pi_Nf(u)\|_{\dot\H^\bb}\le   c_2 r^{-\ff\bb 2} \e^{-\ff{\ll_1}2 r} \big(1+\|u\|_{L^{4m+2}(\mu)}^{2m+1}\big),\ \ r>0.$$
Combining this with \eqref{Z1}, \eqref{NM0} and \eqref{NM}, we find  constants  $ c_3,c_4,c_5>0$ such that
 \beq\label{BN} \beg{split}& \|u_t^N\|_{L^p(\OO; \dot\H^\bb)} \le c_3+ \|u_0\|_{\dot\H^\bb} +\ff {c_3}  \vv \int_0^t (t-s)^{-\ff \bb 2}\e^{-\ff{\ll_1}2(t-s)} \Big(\E\big[1+\|u_s^N\|_{L^{4m+2}(\mu)}^{(2m+1)p}\big]\Big)^{\ff 1 p} \d s \\
   &\le c_3+ \|u_0\|_{\dot\H^\bb} +\ff {c_4}  \vv \int_0^t (t-s)^{-\ff \bb 2} \e^{-\ff{\ll_1}2 (t-s)} \Big(\E\big[\|u_s^N\|_{\dot\H^{\aa_{m,d}}}^{(2m+1)p}\big]\Big)^{\ff 1 p} \d s \\
   &\le   \|u_0\|_{\dot\H^\bb}+c_5 \vv^{-1}(1+\|u_0\|_{\dot\H^{\aa_{m,d}}}^{2m+1}) +c_5    \vv^{-1-\ff{{\aa_{m,d}}(2m+1)}2}(1+\|u_0\|_\H^{2m+1})\\
   &=  \|u_0\|_{\dot\H^\bb} +c_5 \vv^{-1}(1+\|u_0\|_{\dot\H^{{\aa_{m,d}}}}^{2m+1}) +c_5    \vv^{- \ff{md+2}2}(1+\|u_0\|_\H^{2m+1}).\end{split}\end{equation}
This completes the proof.

\end{proof}

Finally,  we estimate moments of    $\|u^N_t\|_{L^\infty(\mu)}$.

\beg{lem}\label{L2'}   Assume {\bf (A)}, {\bf (F)} and {\bf (W)} with $\aa_{m,d}:=\ff{md}{2m+1}\le 1$.
\beg{enumerate}
\item[$(1)$] Let $d\in (0,2)$ and $\eqref{Q}$  hold  for some  $\aa\in \big[\aa_{m,d},1\big]\cap\big(\ff{d-2}2,1\big]$.  Then for any $p\ge 2$ and $\beta\in (\frac d2,1]$, there exists a constant $c>0$ such that
\beq\label{I1} \sup_{t\ge 0, N\in \bar\N}\|u_t^N\|_{L^p(\OO; L^\infty(\mu))}\le c\|u_0\|_{\dot\H^\beta} + c\vv^{-\ff{\beta} 2}(1+\|u_0\|_{\H}).\end{equation}
 \item[$(2)$]   Let $d\in [2,2+m^{-1}]\cap [2, 4)$  and   $\eqref{Q}$ hold  for some  $\aa\in \big[\aa_{m,d},1\big]\cap\big(\ff{d-2}2,1\big]$.   Then for any $p\ge 2$ and any $\bb\in (\ff d 2,2]\cap [\alpha,1+\alpha]$, there exists a constant $c>0$ such that
\beq\label{I2} \beg{split} \sup_{t\ge 0, N\in \bar\N}\|u_t^N\|_{L^p(\OO; L^\infty(\mu))}
 \le c\|u_0\|_{\dot\H^\bb} +c \vv^{-1} \|u_0\|_{\dot\H^{\aa_{m,d}}}^{2m+1}+c    \vv^{- \ff{md+2}2}(1+\|u_0\|_\H^{2m+1}).\end{split}\end{equation}
\item[$(3)$]
Let $d\in [2,2+m^{-1}]\cap [2, 4).$       If  there exists $q>d$ such that
\beq\label{NC} \sup_{i\ge 1}\|e_i\|_{L^q(\mu)}<\infty,
\end{equation} then for any $p\ge 2$ and any $\beta\in (\frac d2,2)$, there exists a constant  $c>0$ such that
$\eqref{I2}$ holds.
\end{enumerate}
\end{lem}

\beg{proof}
(a) By \cite[Theorem 3.3.15]{Wbook}, together with $\ll_1>0$, \eqref{NS} implies that for some constant $c\ge 1$,
$$\|\e^{-At} \|_{L^1(\mu)\to L^\infty(\mu)}\le c t^{-\ff d 2}\e^{-\ll_1 t/2},\ \ t>0,$$
which together with the contraction of $\e^{-At}$ in $L^p(\mu)$ for all $p\ge 1$ and the Riesz-Thorin interpolation theorem, leads to
\beq\label{Con-est}\|\e^{-At}\|_{L^p(\mu)\to L^q(\mu)}\le c t^{-\ff {d(q-p)} {2pq} } \e^{-\ff{\ll_1t}2},\ \ t>0,\ 1\le p\le q\le \infty.\end{equation}
Noting that  $A^{-\beta} =\ff 1 {\Gamma(\beta)} \int_0^\infty s^{\beta-1} \e^{-s A}\d s$ for $\beta>0$, this  implies
\beq\label{Sob1} \beg{split}& \|u\|_{L^\infty(\mu)} = \|A^{-\ff \beta 2} A^{\ff\beta 2 } u\|_{L^\infty(\mu)}\le \ff {c}{\Gamma(\aa/2)} \int_0^\infty s^{\ff\beta 2 -1} s^{-\ff d 4} \e^{-\ll_1 s/2}  \|A^{\ff\beta 2 } u\|_\H\d s\\
&= \ff {c}{\Gamma(\beta/2)} \int_0^\infty s^{\ff\bb 2 -1} s^{-\ff d 4}\e^{-\ll_1 s/2} \| u\|_{\dot\H^\beta}  \d s =:c(\beta)\| u\|_{\dot\H^\beta}<\infty,\ \ \beta>\ff d 2,\ u\in \dot\H^\aa.\end{split}\end{equation}
Combining this with Lemma \ref{L2} (2)-(3) we derive \eqref{I1} for some constant $c>0.$

(b)
Since
$$\lim_{q\downarrow d} \lim_{\theta\downarrow \ff{d}{2q}} \Big(2\theta +\ff{d(q-2)}{2q}-1\Big)=\ff{d-2}{2}<\aa,$$
there exist constants $q,\theta$ such that
\beq\label{XDD} q>d,\ \ \ \ \ff{d}{2q}<\theta<\ff 1 2,\ \  \ \ 2\theta +\ff{d(q-2)}{2q}-1\le \aa.\end{equation}
So, \eqref{Q} implies
\beq\label{BB'} \sum_{i=1}^\infty q_i \ll_i^{\ff{d(q-2)}{2q} +2\theta-1}<\infty.\end{equation}
Let
\beq\label{Y1}  Y_{\theta}(t)=\int_0^{t}(t-s)^{-\theta} \pi_N  \e^{-A(t-s)} \d W_s,\ \ t\ge 0.\end{equation}
By \eqref{Con-est}, we find a constant $c_1>0$ such that
$$\|e_i\|_{L^q(\mu)}=\inf_{r>0} \e^{r\ll_i}\|\e^{-Ar}  e_i\|_{L^q(\mu)}\le c \inf_{r>0} \big[\e^{r\ll_i} r^{-\ff {d(q-2)}{4q}}\big]\le c_1 \ll_i^{\ff{d(q-2)}{4q}}.$$
Combining this with \eqref{BB'},  \eqref{Y1},    Burkholder inequality (see e.g.  \cite[Theorem 1.2]{Vanzhu11} since $L^q(\mu),q\ge 2$, is 2-smooth and type 2 Banach space),   we find constants $c_1, c_2 >0$ such that
\beq\label{BB''} \beg{split} &\|Y_{\theta}(t)\|_{L^p(\OO;L^q(\mu))}=\bigg\|\sum_{i=1}^N   \bigg(\int_0^t (t-s)^{-\theta} \e^{-\ll_i(t-s)}q_i^{\ff 1 2}\d W_i(s)\bigg)e_i \bigg\|_{L^p(\OO;L^q(\mu))}\\
& \le c_1 \bigg(\sum_{i=1}^N q_i \|e_i\|_{L^q(\mu)}^2 \int_0^t(t-s)^{-2\theta} \e^{-2\ll_i(t-s)}  \d s\bigg)^{\ff 1 2}\\
&\le   c_2 \bigg(\sum_{i=1}^\infty q_i \ll_i^{\ff{d(q-2)}{2q} +2\theta-1}  \Big)^{\ff 1 2}
 <\infty.\end{split}\end{equation}
 By the stochastic Fubini theorem \cite{Leon90}, we have
$$
Z^N(t)=\frac {\sin (\theta \pi) }{\pi} \int_0^{t}(t-s)^{\theta-1}  \e^{-A(t-s)} Y_{\theta}(s)\d s,\ \ t\ge 0.
$$
So, by \eqref{Con-est}, $\theta>\ff d{2q}$ in \eqref{XDD} and \eqref{BB''}, we find   constants $c_3,c_4>0$ such that for any $t\ge 0$ and $N\in\bar\N$,
$$
  \|Z^N (t)\|_{L^p(\Omega;   L^{\infty}(\mu)) }
 \le c_3 \int_0^t (t-s)^{\theta-1-\ff{d}{2q}} \e^{-\ll_1(t-s)/2}  \|Y_{\theta}(s)\|_{L^p(\Omega;L^{q}(\mu))}\d s \le c_4.
$$
Noting that $\e^{-At}$ is contractive in $L^\infty(\mu)$, $\|u_0^N\|_{\dot\H^\beta}\le \|u_0\|_{\dot\H^\beta}$ for $\beta\ge 0$, and $d<4$,
 by   combining this with \eqref{Z1}, \eqref{Con-est}, \eqref{F1},    \eqref{PB},  and  \eqref{*2'},   we find   constants $c_5,c_6>0$ such that
  \beg{align*} &\|u_t^N\|_{L^p(\OO; L^\infty(\mu))}\le \|  u_0^N\|_{L^\infty(\mu)}
  + \ff{\kk_1}\vv \int_0^t (t-s)^{-\ff {d} 4} \e^{-\ff{\ll_1}2 (t-s) } \|1+(u_s^N)^{2m+1}\|_{L^p(\OO;\H)}\d s+c_4\\
 & \le \|  u_0^N\|_{L^\infty(\mu)}
  + \ff{c_5}\vv \int_0^t (t-s)^{-\ff {d} 4} \e^{-\ff{\ll_1}2(t-s)} \big[\|u_0\|_{\dot \H_0^{{\aa_{m,d}}}}+\vv^{-\ff{\aa_{m,d}} 2}(1+\|u_0\|_\H)\big]^{2m+1}\d s\\
 &\le c_6 \|  u_0\|_{\dot \H^{\beta}}+ c_6 \vv^{-1}     \|u_0\|_{\dot\H^{{\aa_{m,d}}}}^{1+2m} +c_6  \vv^{-1-\ff{md}{2}} (1+\|u_0\|_\H^{1+2m}).
  \end{align*} So, \eqref{I2} in (2) holds.

 (c)
  Thanks to \eqref{NC}, by $2\theta<1$ we can modify \eqref{BB''}    as
$$\|Y_{\theta}(t)\|_{L^p(\OO;L^\infty(\mu))}\le c_1 \bigg(\sum_{i=1}^N q_i \|e_i\|_{L^q(\mu)}^2 \int_0^t(t-s)^{-2\theta} \e^{-2\ll_i(t-s)}  \d s\bigg)^{\ff 1 2}
\le c_2 \bigg(\sum_{i=1}^N q_i  \bigg)^{\ff 1 2} <\infty.$$  Then  \eqref{I2} in (3) still holds according to the proof in the last step (b).

 \end{proof}

Note that the condition \eqref{NC} with $q=\infty$ holds for eigenfunctions of the Laplacian on the torus  $\mathbb T^d$.

\subsection{Estimates on $u_t^{N,\tau}$}

Let $N\in \mathbb N$. Since for any $\sigma>0$, $\|\cdot\|_{\dot\H^\sigma}$ and  $\|\cdot\|_\H$ are equivalent on $\H_N$,  $u_t^{N,\tau}$ defined in \eqref{TF} is a continuous process on
  $\dot\H^\sigma$ so that   $\P$-a.s.
  $$T_{R,\sigma}:= \inf\big\{t\ge 0: \|u_t^{N,\tau}\|_{\dot\H^\sigma}\ge R\big\}\to\infty\ \text{as}\ R\to\infty.$$
  By \eqref{F1} and the definition of $\Theta_{\tau,\sigma}$ in \eqref{def-thetat}, we have
 \beg{align} \label{TA} & |\Theta_{\tau,\sigma}(u)|\le \kk_1 \vv^{-1} (1+|u|^{2m+1}),\\
     \label{TA'} &|\Theta_{\tau,\sigma}(u) - \Theta_0(u)|\le \kk_1  \vv^{-1}\big(1+|u|^{2m+1}\big)  \tau\|u\|_{\dot \H^{\sigma}}^{2m+1}.\end{align}
  
 \beg{lem}\label{LN0} Assume  {\bf (A)}, {\bf (F)} and {\bf (W)} with $\aa_{m,d}:=\ff{md}{2m+1}\le 1,$ and let  $\eqref{Q}$ hold for some  $\aa\in [\aa_{m,d},1].$
 Then the following assertions hold.
 \beg{enumerate}
  \item[$(1)$]  For any   $p\in [2,\infty)$, any $\sigma\in (0,2)$  and   $q\in [\frac 12,\ff d{(d-2\sigma)^+}) $ so that
  \begin{align}\label{def-gamma}
  \dd_{d,q,\sigma}:= \ff{d+ (2\sigma-d)q}{4q}\in (0,1),\ \ \ \ \gg_{d,q,\sigma}:=\ff{d+(4-d)q}{4q}\in (0,1],
  \end{align}
   there exists a constant $c>0$ such that for any $R\in [1,\infty)$,
  \beq\label{LN0-11}  \|1_{\{T_{R,\sigma}\ge t\}}(u_t^{N,\tau}- u_{\tau_{k(t)}}^{N,\tau})\|_{L^p(\OO; L^{2q}(\mu))}  \le c\Big(\tau^{\dd_{d,q,\sigma}} R +\tau^{\gg_{d,q,\sigma}}  \vv^{-1} R^{2m+1}+ \ss\tau\Big).\end{equation}
  In particular,  for $q=1$ we have $\dd_{d,q,\sigma}=\ff \sigma 2$ and $\gg_{d,q,\sigma}=1$ so that
   \beq\label{LN0-1}  \|1_{\{T_{R,\sigma}\ge t\}}(u_t^{N,\tau}- u_{\tau_{k(t)}}^{N,\tau})\|_{L^p(\OO; L^{2}(\mu))}  \le c\Big(\tau^{\ff \sigma 2} R +\tau   \vv^{-1} R^{2m+1}+ \ss\tau\Big).\end{equation}
   \item[$(2)$]
For any $p\in [2,\infty)$ and any $\si\in [\aa_{m,d},2)$,  there exists a constant $c>0$ such that for any $R\in [1,\infty)$,
 \beq\label{LN0-2} \beg{split}  &\|1_{\{T_{R,\sigma}\ge t\}}u_t^{N,\tau}\|_{L^p(\OO; \H)}
    \le   \|u_0\|_\H\\
    &+  c \Big(1+ \tau^{\ff\sigma 4}R^{m+1}  +\tau^{\frac 12}\vv^{-\ff 12}R^{2m+1}+  \tau^{\frac 12} R^{\frac {4m+3}2} \Big)^{\ff{1+2m/p}{m+1}}.\end{split}\end{equation}
 \item[$(3)$]
When $\aa_{m,d}<1  $ and
   $\sigma\in [\alpha,1],$  for any   $q_{d,\sigma} $ in \eqref{QD} and  any  $ p\in [2,\infty), $ there exists a   constant
$c>0$ such that
 \beq\label{LN0-3} \beg{split} &\|1_{\{T_{R,\sigma}\ge t\}}u_t^{N,\tau}\|_{L^p(\OO; \dot\H^\sigma)}\le \|u_0\|_{\dot\H^\sigma}\\
 &+c\Big(1+\vv^{-\frac 12}\tau^{\dd_{d,q_{d,\sigma},\sigma}} R^{2m+1} +\tau^{\gg_{d,q_{d,\sigma},\sigma}}  \vv^{-\frac 32} R^{4m+1}+ \vv^{-\frac 12} \ss\tau R^{2m} +\tau \vv^{-\frac 12}  R^{4m+2}\Big)\\
&+c\vv^{-\frac {\sigma}2}\Big( \|u_0\|_\H+   \Big(1+ \tau^{\ff\sigma 4}R^{m+1}  +\tau^{\frac 12}\vv^{-\frac 12}R^{2m+1}+ \tau^{\frac 12}  R^{\frac {4m+3} 2}\Big)^{\ff{1+2m/p}{m+1}}\Big).
   \end{split}\end{equation}
   Consequently, there exists a constant $c>0$ such that
   \beq\label{X1} \|1_{\{T_{R,\sigma}\ge t\}}u_t^{N,\tau}\|_{L^p(\OO; \dot\H^\sigma)}
  \le   c \vv^{-\ff \sigma 2}  (1+\|u_0\|_{\dot\H^\sigma})
    \end{equation} holds  for $\gg_2$ in $\eqref{DD}$ and
   \beq\label{TT0} R:= \vv^{\frac 3{2(4m+2)}}\tau^{-\frac {1}{(4m+2)\gg_2}},\ \ \ \tau\le \vv^{  \ff{3\gg_2}2}.\end{equation}
    \item[$(4)$]
 When $\aa_{m,d}=1$ and $ \sigma\in (1,2),$
    for any  $ p\in [2,\infty), $  there exists $c>0$ such that
 \beq\label{LN0-4} \beg{split}
 &\|1_{\{T_{R,\sigma}\ge t\}}u_t^{N,\tau}\|_{L^p(\OO; \dot\H^{\sigma})}\le c(1+\|u_0\|_{\dot\H^{\sigma}})+c\vv^{-\frac { 2m+3}2} (1+\|u_0\|_{\dot\H^{1}})^{2m+1}
   \end{split}\end{equation}holds for  $\gg_3$   in $\eqref{DD}$   and
      \beq\label{TT1} R:= \vv^{\frac 3{2(4m+2)}}\tau^{-\frac {1}{(4m+2)\gg_3}},\ \ \ \tau\le \vv^{  \ff{3\gg_3}2}.\end{equation}
  \end{enumerate}\end{lem}

 \beg{proof} (a)  By \eqref{TF}, we have
 \beq\label{M1} \beg{split} &u_t^{N,\tau}-u_{\tau_{k(t)}}^{N,\tau}= I_1+I_2+I_3,\\
 &I_1:=\big(\e^{-A(t-\tau_{k(t)})}-1\big) u_{\tau_{k(t)}}^{N,\tau}, \\
&I_2:=  \int_{\tau_{k(t)}}^t \pi_N \e^{-A(t-s)} \Theta_{\tau,\sigma} (u_{\tau_{k(s)}}^{N,\tau}) \d s,\\
&I_3:=  \int_{\tau_{k(t)}}^t \pi_N \e^{-A(t-s)}\d W_s. \end{split}
 \end{equation}
Since    $q\in [\frac 12,\ff d{(d-2\sigma)^+})$, we have
\beq\label{M2}  \bb:= \ff{d{(q-1)^+}} {2q}\le\sigma,\end{equation}
so that  \eqref{S} implies
 \beq\label{PO1}  \|u\|_{L^{2q}(\mu)}=\|u\|_{L^{\ff{2d}{d-2\bb}} (\mu)}\le c(\sigma) \|u\|_{\dot\H^\beta}  \le c(\sigma) \|u\|_{\dot\H^\sigma}.\end{equation}
  Combining this with the fact that
 $$|\e^{-s}-1|\le s^\theta,\ \ \theta\in [0,1],\ s\ge 0,$$
 we obtain
 \beq\label{M3}\beg{split}
 \|I_1\|_{L^{2q}(\mu)} &\le c(\sigma) (t-\tau_{k(t)})^{\ff{\sigma-\bb}2 } \big\| A^{\ff{\sigma-\bb}2} A^{\ff\bb 2} u_{\tau_{k(t)}}^{N,\tau}\big\|_{\H}  \\
 &\le c(\sigma) \tau^{\ff{\sigma-\bb}2} R= c(\sigma) \tau^{\dd_{d,q,\sigma}} R,\ \ \ \ t\le T_{R,\sigma}.\end{split}\end{equation}
 Next,  by \eqref{TA}, \eqref{BN0} and \eqref{PO1},
 we find   constants $c_1,c_2>0$ such that
 \beg{align*} \|I_2\|_{L^{2q}(\mu)}
 &\le c_1   \int_{\tau_{k(t)}}^t   (t-s)^{-\ff\bb 2} \e^{-\ff{\ll_1} 2(t-s)} \| \Theta_{\tau,\sigma} (u_{\tau_{k(s)}}^{N,\tau})\|_\H \d s \\
 &\le c_2 \vv^{-1} \int_{\tau_{k(t)}}^t  (t-s)^{-\ff\bb 2} \e^{-\ff{\ll_1} 2(t-s)}\big(1+\| u_{\tau_{k(s)}}^{N,\tau}\|_{L^{4m+2}(\mu)}^{2m+1}\big)  \d s.\end{align*}
 By \eqref{PB} and $\sigma\ge \aa_{m,d}$, we find a constant $c_3>0$ such that
 $$  \| u_{\tau_{k(s)}}^{N,\tau}\|_{L^{4m+2}(\mu)} \le c_3 \| u_{\tau_{k(s)}}^{N,\tau}\|_{\dot\H^\sigma} \le   c_3 R,\ \ t\le T_{R,\sigma}.$$
 Hence,  there exists a constant $c_4>0$ such that
 \beq\label{M4} \|I_2\|_{L^{2q}(\mu)}\le c_4   \tau^{1-\ff \bb 2} \vv^{-1} R^{2m+1}= c_4   \tau^{\gg_{d,q,\sigma}} \vv^{-1} R^{2m+1}.\end{equation}
 Finally, by $\bb\le \sigma$,  \eqref{Q} and \eqref{PO1}, we find constants $c_5,c_6>0$ such that
 \beg{align*} &\E \bigg[\bigg\|\int_{\tau_{k(t)}}^t \pi_N \e^{-A(t-s)} \d W_s\bigg\|_{L^{2q}(\mu)}^p\bigg] \le c(\bb)^p \E \bigg[\bigg\|\int_{\tau_{k(t)}}^t \e^{-A(t-s)} \d W_s\bigg\|_{\dot\H^\bb}^p\bigg]\\
 &\le c_5 \bigg(\int_{\tau_{k(t)}}^t \sum_{i=1}^\infty q_i\ll_i^{\bb  } \d s\bigg)^{\ff p 2}\le c_6 \tau^{\ff p 2}.\end{align*}
 Combining this with  \eqref{M3}-\eqref{M4},
 we derive  \eqref{LN0-11} for some constant $c>0$.

(b)  By It\^o's formula, for any $p\in [2,\infty)$ we find a constant $c_0>0$ and a martingale $M_t$ such that
\beq\label{U2}\beg{split}  &\d \|u_t^{N,\tau}\|_\H^p +\d M_t
\le p\|u_t^{N,\tau}\|_\H^{p-2}\big(c_0 + J_1+J_2+J_3\big)\d t,\\
&J_1:=-\vv^{-1} \<u_t^{N,\tau}, f(u_t^{N,\tau})\>_\H,\\
& J_2:=-\vv^{-1}\<u_t^{N,\tau}, f(u_{\tau_{k(t)}}^{N,\tau})- f(u_t^{N,\tau})\>_\H,\\
&J_3:= \<u_t^{N,\tau}, \Theta_{\tau,\sigma}(u_{\tau_{k(t)}}^{N,\tau})- \Theta_0(u_{\tau_{k(t)}}^{N,\tau})\>_\H- \|u_t^{N,\tau}\|_{\dot\H^1}^2.\end{split}
\end{equation}
By \eqref{F2}, we find a constant $c_1>0$ such that
$$J_1\le \vv^{-1}\big(\kk_1-  \kk_2 \|u_t^{N,\tau}\|_{L^{2m+2}(\mu)}^{2m+2}\big)\le c_1\vv^{-1} -c_1\vv^{-1} \|u_t^{N,\tau}\|_{\H}^{2m+2}.$$
By \eqref{F4}, the Schwarz inequality and  \eqref{PB}, we find   constants $c_2,c_3>0$ such that
\beg{align*}J_2&\le \vv^{-1} \kk_1\big\|u_t^{N,\tau}(u_t^{N,\tau}-u_{\tau_{k(t)}}^{N,\tau})
 (1+ |u_t^{N,\tau} |^{2m} +  |u_{\tau_{k(t)}}^{N,\tau}|^{2m})\big\|_{L^1(\mu)}\\
&\le c_2  \vv^{-1}  \|u_t^{N,\tau}-u_{\tau_{k(t)}}^{N,\tau}\|_{\H} \big(1+\|u_t^{N,\tau}\|_{L^{4m+2}(\mu)}^{2m+1}+ \|u_{\tau_{k(t)}}^{N,\tau}\|_{L^{4m+2}(\mu)}^{2m+1}\big)\\
&\le c_3   \vv^{-1}\|u_t^{N,\tau}-u_{\tau_{k(t)}}^{N,\tau}\|_{\H}  (1+R^{2m+1}),\ \ \  t\in [0,T_{R,\sigma}].
 \end{align*}
By \eqref{TA'}, \eqref{S} and $\aa_{m,d}<1$, we find a constant $c_4>0$ such that
  \beg{align*} \|\Theta_{\tau,\sigma}(u)-\Theta_0(u)\|_{\H}&\le \kk_1\vv^{-1} \big(1+\|u\|_{L^{4m+2}(\mu)}^{2m+1}\big) \tau\|u\|_{\dot\H^\si}^{2m+1} \\
 &\le c_4   \tau  \vv^{-1} (1+\|u\|_{\dot \H^{\sigma}}^{4m+2}). \end{align*}  Combining this with H\"older's inequality, we obtain 
\beg{align*} J_3 \le c_4 \tau \vv^{-1} \|u_t^{N,\tau}\|_{\H}  (1+\|u_t^{N,\tau}\|_{\dot \H^{\sigma}}^{4m+2}) -\|u_t^{N,\tau}\|_{\dot\H^1}^2.\end{align*}
By  the above estimates  on $J_1,J_2,J_3$,    \eqref{U2} and noting that for any constants $p_1>p_2\ge 0$ and $\dd>0$,
there exists a constant $c(p_1,p_2,\dd)>0$ such that
\beq\label{U3}a^{p_2}b\le \dd a^{p_1} + c(p_1,p_2,\dd) b^{\ff{p_1}{p_1-p_2}},\ \ a,b\ge 0,\end{equation}
we find   constants  $c_5,c_6>0$ such that  for $t\le T_{R,\sigma}$,
\beg{align*}&\d \|u_t^{N,\tau}\|_\H^p +\d M_t\\
&\le p\|u_t^{N,\tau}\|_\H^{p-2}
\vv^{-1}  \Big(c_5  - c_1  \|u_t^{N,\tau}\|_{\H}^{2m+2}+  c_3  \|u_t^{N,\tau}-u_{\tau_{k(t)}}^{N,\tau}\|_{L^{2}(\mu)}  R^{2m+1}\\
&\quad+c_4 \tau \|u_t^{N,\tau}\|_{\H}  (1+\|u_t^{N,\tau}\|_{\dot \H^{\sigma}}^{4m+2}) \Big)\d t \\
 &\le \vv^{-1} \Big(- c_1  \|u_t^{N,\tau}\|_{\H}^{p}
  + c_6 \big(1+\|u_t^{N,\tau}-u_{\tau_{k(t)}}^{N,\tau}\|_{L^{2}(\mu)}R^{2m+1}+   \tau R^{4m+3}\big)^{\ff {p+2m}{2m+2}}\Big)\d t.
 \end{align*}
 Hence, by \eqref{LN0-1} and $R\ge 1\ge\tau\lor\vv\lor\sigma$,
   we find some constants $c_7,c_8>0$ such that
 \beg{align*} &\E\big[1_{[0,T_{R,\sigma}]}(t) \|u_t^{N,\tau}\|_\H^p\big] &\\
 &\le  \|u_0\|_\H^p+     \ff{c_6}\vv  \int_0^t\e^{-\ff{t-s}{c_1 \vv}  }  \E\Big[1_{\{s\le T_{R,\sigma}\} } \big(1+\|u_s^{N,\tau}-u_{\tau_{k(s)}}^{N,\tau}\|_{L^{2}(\mu)}R^{2m+1}+ \tau R^{4m+3}\big)^{\ff {p+2m}{2m+2}} \Big]\d s\\
 &\le \|u_0\|_\H^p+c_7 \big(1+\tau^{\ff\sigma 2} R^{2m+2} +\tau\vv^{-1}R^{4m+2}+  \ss\tau\, R^{2m+1}+ \tau R^{4m+3}\big)^{\ff {p+2m}{2m+2}}\\
 &\le \|u_0\|_\H^p+c_8 \big(1+\tau^{\ff\sigma 2} R^{2m+2} +\tau\vv^{-1}R^{4m+2} + \tau R^{4m+3}\big)^{\ff {p+2m}{2m+2}}.\end{align*}
This implies  \eqref{LN0-2}   for some constant $c>0.$

(c) By \eqref{Q} and It\^o's formula, we find a constant $c_0>0$ and a martingale $M_t$  such that
\beq\label{*4} \beg{split} &\d\|u_t^{N,\tau}\|_{\dot\H^\sigma}^p+\d M_t
 \le  p \|u_t^{N,\tau}\|_{\dot\H^\sigma}^{p-2} \big(c_0- \|u_t^{N,\tau}\|_{\dot\H^{1+\sigma}}^2+K_1+K_2+K_3\big)\d t,\\
&K_1:=-\vv^{-1}\<u_t^{N,\tau}, f(u_t^{N,\tau})\>_{\dot\H^\sigma},\\
& K_2:=-\vv^{-1}\<u_t^{N,\tau}, f(u_{\tau_{k(t)}}^{N,\tau})- f(u_t^{N,\tau})\>_{\dot\H^\sigma},\\
&K_3:=-\<u_t^{N,\tau}, \Theta_{\tau,\sigma}(u_{\tau_{k(t)}}^{N,\tau})- \Theta_0(u_{\tau_{k(t)}}^{N,\tau})\>_{\dot \H^{\sigma}}.
\end{split}
\end{equation}
By \eqref{F3}  and $\sigma \in (0,1]$, we obtain
\beq\label{K1}K_1\le -\kk_1  \vv^{-1} \|u_t^{N,\tau}\|_{\dot\H^\sigma}^2.\end{equation}
Let
$$p_0:=1\lor \frac {2d}{d+2(1-\sigma)}\le 2,\ \ \  p_1=\frac {2d}{2m(d-2\sigma)^+}.$$
By  $q_{d,\si}\in[\hat q,\ff d{(d-2\si)^+})$ in \eqref{QD}, we have
\begin{align}\label{hold-cond}
\frac 1{p_0}\ge \frac 1{p_1}+\frac 1{2q_{d,\sigma}},\ \ \ff{p_0}{p_0-1}\le \ff{2d}{(d-2(1-\si))^+}.
\end{align}
By  \eqref{S}, we find a constant $c_1>0$ such that
 \beq\label{*QP} \|u\|_{L^{2mp_1}(\mu)}\le c(\sigma) \|u\|_{\dot\H^\sigma},\ \ \|u\|_{L^{\ff{p_0}{p_0-1}}(\mu)}\le c_1\|u\|_{\dot\H^{1-\sigma}},\end{equation}  where the second inequality and H\"older's inequality imply
 \beq\label{*PQ} \|u\|_{\dot\H^{\sigma-1}}=\sup_{\|v\|_\H\le 1} \<u,A^{\sigma-1}v\>_\H\le  \|u\|_{L^{p_0}(\mu)} \sup_{\|v\|_\H\le 1} \|A^{\sigma-1}v\|_{L^{\ff{p_0}{p_0-1}}(\mu)}\le c_1\|u\|_{L^{p_0}(\mu)}.\end{equation}
 On the other hand,  in \eqref{def-gamma} with $q=q_{d,\sigma}$  we  simply denote
 $$\dd_{d,\sigma}=\dd_{d,q_{d,\sigma},\sigma}:= \ff{d+(2\sigma-d)q_{d,\sigma}}{4 q_{d,\sigma}},\ \ \
\gg_{d,\sigma}:= \ff{d+(4-d)q_{d,\sigma}}{4 q_{d,\sigma}}.$$
 By \eqref{LN0-11} with $2q=2q_{d,\sigma}<\frac {2d}{(d-2\sigma)^+},$ we obtain
\begin{align}\label{time-con-est}
\|u_t^{N,\tau}-u_{\tau_{k(t)}}^{N,\tau}\|_{L^{2q_{d,\sigma}}(\mu)}\le c\Big(\tau^{\dd_{d,\sigma}} R +\tau^{\gg_{d,\sigma} } \vv^{-1} R^{2m+1}+ \ss\tau\Big).
\end{align}
Combining this with \eqref{F4}, \eqref{*QP}, \eqref{*PQ} and   H\"older's inequality due to \eqref{hold-cond},
we find constants $c_2,c_3>0$ such that
\beg{align*} K_2&\le \vv^{-1}\|u_t^{N,\tau}\|_{\dot\H^{\sigma+1}}\|f(u_t^{N,\tau})-f(u_{\tau_{k(t)}}^{N,\tau})\|_{\dot\H^{\sigma-1}}\\
&\le \vv^{-1}c_1\kk_1\|u_t^{N,\tau}\|_{\dot\H^{\sigma+1}}\big\|(u_t^{N,\tau}-u_{\tau_{k(t)}}^{N,\tau})
(1+|u_t^{N,\tau}|^{2m} +|u_{\tau_{k(t)}}^{N,\tau}|^{2m}\big)\big\|_{L^{p_0}(\mu)}\\
&\le c_2 \vv^{-1} \|u_t^{N,\tau}\|_{\dot\H^{\sigma+1}}\|u_t^{N,\tau}-u_{\tau_{k(t)}}^{N,\tau}\|_{L^{2q_{d,\sigma}}(\mu)}
\Big(1+\|u_t^{N,\tau}\|_{L^{2mp_1}(\mu)}^{2m}
+\Big\|u_{\tau_{k(t)}}^{N,\tau}\Big\|_{L^{2mp_1}(\mu)}^{2m}\Big)\\
&\le c_3\vv^{-1}\|u_t^{N,\tau}\|_{\dot\H^{\sigma+1}}\Big(\tau^{\dd_{d,\sigma}} R^{2m+1} +\tau^{\gg_{d,\sigma}}  \vv^{-1} R^{4m+1}+ \ss\tau R^{2m}\Big) ,
\ \ t\le T_{R,\sigma}.\end{align*}
Moreover, by \eqref{TA'}, \eqref{*PQ}, $p_0=1\lor\ff{2d}{d+2(1-\sigma)}$ and \eqref{S},
   we find constants $c_4,c_5>0$ such that
\begin{align*}
K_3 &\le \|u_t^{N,\tau}\|_{\dot \H^{\sigma+1}} \|\Theta_{\tau,\sigma}(u_{\tau_{k(t)}}^{N,\tau})- \Theta_0(u_{\tau_{k(t)}}^{N,\tau})\|_{\dot \H^{\sigma-1}}\\
&\le  c_4   \vv^{-1} \|u_t^{N,\tau}\|_{\dot \H^{\sigma+1}}
\tau \Big(1+\Big\|u_{\tau_{k(t)}}^{N,\tau}\Big\|^{2m+1}_{L^{(2m+1) p_0}(\mu)}\Big)\Big\|u_{\tau_{k(t)}}^{N,\tau}\Big\|_{\dot \H^{\sigma}}^{2m+1}\\
&\le c_5 \tau \vv^{-1} \|u_t^{N,\tau}\|_{\dot \H^{\sigma+1}} R^{4m+2},
\ \ t\le T_{R,\sigma}.
\end{align*}
This together with   \eqref{K1} and the estimates of $K_2,K_3$, yields
\beq\label{OI}\beg{split} &c_0- \|u_t^{N,\tau}\|_{\dot\H^{1+\sigma}}^2+K_1+K_2+K_3\\
  &\le   c_0- \|u_t^{N,\tau}\|_{\dot\H^{1+\sigma}}^2-
 \kk_1  \vv^{-1} \|u_t^{N,\tau}\|_{\dot\H^\sigma}^2\\
&+ c_3\vv^{-1}\|u_t^{N,\tau}\|_{\dot\H^{\sigma+1}}\Big(\tau^{\dd_{d,\sigma}} R^{2m+1}
+\tau^{\gg_{d,\sigma}}  \vv^{-1} R^{4m+1}+ \ss\tau R^{2m}\Big) \\
&+c_5 \tau \vv^{-1} \|u_t^{N,\tau}\|_{\dot \H^{\sigma+1}} R^{4m+2},
\ \ t\le T_{R,\sigma}.
 \end{split}\end{equation}
 Noting that
$$\|u\|_{\dot\H^{\sigma}}\le \|u\|_{\dot\H^{\sigma+1}}^{\ff{\sigma}{\sigma+1}}\|u\|_\H^{\ff{1}{\sigma+1}}, $$
we find a constant $c_6>0$ such that
\beg{align*} &   c_0- \|u_t^{N,\tau}\|_{\dot\H^{1+\sigma}}^2+K_1+K_2+K_3 \\
&\le c_0- \frac 12 \|u_t^{N,\tau}\|_{\dot\H^{1+\sigma}}^2-
 \kk_1  \vv^{-1} \|u_t^{N,\tau}\|_{\dot\H^\sigma}^2\\
&+ 2c_3\Big(\vv^{-1} \tau^{\dd_{d,\sigma}} R^{2m+1} +\tau^{\gg_{d,\sigma}}  \vv^{-2} R^{4m+1}+ \vv^{-1}\ss\tau R^{2m}\Big)^2 +2c_5 \big(\tau \vv^{-1}  R^{4m+2}\big)^2\\
&\le c_0-\frac 12 \kk_1  \vv^{-1} \|u_t^{N,\tau}\|_{\dot\H^\sigma}^2+
2c_3\Big(\vv^{-1}\tau^{\dd_{d,\sigma}} R^{2m+1} +\tau^{\gg_{d,\sigma}}  \vv^{-2} R^{4m+1}+\vv^{-1} \ss\tau R^{2m}\Big)^2\\
&+2c_5\big(\tau \vv^{-1}  R^{4m+2}\big)^2+c_6\vv^{-1-\sigma}\|u_t^{N,\tau}\|_{\H}^2,\ \ t\le T_{R,\sigma}.
\end{align*}
Combining this with \eqref{*4},   \eqref{OI} and \eqref{U3}, we find a constant $c_7>0$ such that
\beg{align*} &\d\|u_t^{N,\tau}\|_{\dot\H^\sigma}^p+\d M_t\\
&\le -\frac 12 \kk_1  \vv^{-1}  \|u_t^{N,\tau}\|_{\dot\H^{\sigma}}^p\d t +c_7 \vv^{-1} \big(1+\vv^{-\sigma}\|u_t^{N,\tau}\|_\H^2 +  \vv \big(\tau \vv^{-1}  R^{4m+2}\big)^2 \big)^{\frac p2}\d t \\
  &+c_7\vv^{-1}\Big(\vv^{-\frac 12}\tau^{\dd_{d,\sigma}} R^{2m+1} +\tau^{\gg_{d,\sigma}}  \vv^{-\frac 32} R^{4m+1}+\vv^{-\frac 12} \ss\tau R^{2m}\Big)^p \d t.\end{align*}
By this, and
 $R\ge 1\ge\tau\lor\sigma$,  we find constants $c_8,c_9>0$ such that
\beg{align*} &\E\Big[1_{[0,T_{R,\sigma}]}(t)  \|u_t^{N,\tau}\|_{\dot\H^\sigma}^p\Big]
 \le  \|u_0\|_{\dot\H^\sigma}^p\\
 &+ c_8\vv^{-1}\int_0^t\e^{-\frac 12 \kk_1  \vv^{-1} (t-s)}\E\Big[1_{[0,T_{R,\sigma}]}(s) \big(1+\vv^{-\sigma}\|u_s^{N,\tau}\|_\H^2 + \vv^{-1}  \big(\tau R^{4m+2}\big)^2 \big) ^{\ff p 2}\Big]\d s\\
&+c_8\vv^{-1}\int_0^t\e^{-\frac 12 \kk_1  \vv^{-1} (t-s)} \Big[\vv^{-1}\Big(\tau^{\dd_{d,\sigma}} R^{2m+1} +\tau^{\gg_{d,\sigma}}  \vv^{-1} R^{4m+1}+ \ss\tau R^{2m}\Big)^2\Big]^{\frac p2}\d s\\
  &\le c_9\Big(1+ \vv^{-\frac 12}\tau^{\dd_{d,\sigma}} R^{2m+1} +\tau^{\gg_{d,\sigma}}  \vv^{-\frac 32} R^{4m+1}+ \vv^{-\frac 12}\ss\tau R^{2m} +\tau \vv^{-\frac 12}  R^{4m+2}\Big)^p\\
  &+c_9 \vv^{-\sigma \frac p2-1} \int_0^t \e^{-\frac 12 \kk_1  \vv^{-1} (t-s)} \E[1_{[0,T_{R,\sigma}]}(s)\|u_s^{N,\tau}\|_\H^p]\d s+\|u_0\|_{\dot\H^\sigma}^p.\end{align*}
This, together with \eqref{LN0-2}, implies that for some $c>0,$
\beg{align*}
&\Big\|1_{[0,T_{R,\sigma}]}(t)  \|u_t^{N,\tau}\|_{\dot\H^\sigma}\Big\|_{L^p(\Omega)}\\
&\le  \|u_0\|_{\dot\H^\sigma}+c\Big(1+\vv^{-\frac 12}\tau^{\dd_{d,\sigma}} R^{2m+1} +\tau^{\gg_{d,\sigma}}  \vv^{-\frac 32} R^{4m+1}+ \vv^{-\frac 12} \ss\tau R^{2m} +\tau \vv^{-\frac 12}  R^{4m+2}\Big)\\
&+c\vv^{-\frac {\sigma}2}\Big( \|u_0\|_\H+   \Big(1+ \tau^{\ff\sigma 4}R^{m+1}  +\tau^{\frac 12}\vv^{-\frac 12}R^{2m+1}+ \tau^{\frac 12}  R^{\frac {4m+3} 2}\Big)^{\ff{1+2m/p}{m+1}}\Big).
\end{align*}
Thus, \eqref{LN0-3} holds for some constant $c>0.$

Finally, if \eqref{TT0} holds, then there exists a constant $c'>0$ such that
\begin{align*}
&\vv^{-\frac 12}\tau^{\dd_{d,\sigma}} R^{2m+1} +\tau^{\gg_{d,\sigma}}  \vv^{-\frac 32} R^{4m+1}+ \vv^{-\frac 12} \ss\tau R^{2m} +\tau \vv^{-\frac 12}  R^{4m+2}\\
&+ \tau^{\ff\sigma 4}R^{m+1}  +\tau^{\frac 12}\vv^{-\frac 12}R^{2m+1}+ \tau^{\frac 12}  R^{\frac {4m+3}2} \le c'.
\end{align*}
Combining this with  \eqref{LN0-3}, we derive  \eqref{X1} for some constant $c>0$.

 (d) The condition $\alpha_{m,d}=1$ implies that $d>2.$
Repeating the steps in (c) with $\sigma$ being replaced by  $1$, one has that
\eqref{hold-cond} holds with $p_0=2, p_1=d, 2q_{d,1}:=\frac {2d}{d-2}.$
Since $\sigma>\alpha_{m,d}=1,$ by \eqref{LN0-11} with $2q=\frac {2d}{d-2}<\frac {2d}{(d-2\sigma)^+},$ we obtain
\begin{align}\label{time-con-est}
\|u_t^{N,\tau}-u_{\tau_{k(t)}}^{N,\tau}\|_{L^{2q_{d,1}}(\mu)}\le c\Big(\tau^{\dd_{d,q_{d,1},\sigma}} R +\tau^{\gg_{d,q_{d,1},\sigma} } \vv^{-1} R^{2m+1}+ \ss\tau\Big).
\end{align}
Then the estimations of $K_1,K_2,K_3$ in \eqref{*4} follow similar to those in (c). Thus, there exists a constant   $c_1>0 $ such that
   \beg{equation}\begin{split}\label{H1-pri} &\|1_{\{T_{R,\sigma}\ge t\}}u_t^{N,\tau}\|_{L^p(\OO; \dot\H^{1})}\le \|u_0\|_{\dot\H^{1}}\\
 &+c_1\Big(1+\vv^{-\frac 12}\tau^{\dd_{d,q_{d,1},\sigma}} R^{2m+1} +\tau^{\gg_{d,q_{d,1},\sigma}}  \vv^{-\frac 32} R^{4m+1}+ \vv^{-\frac 12} \ss\tau R^{2m} +\tau \vv^{-\frac 12}  R^{4m+2}\Big)\\
&+c\vv^{-\frac {1}2}\Big( \|u_0\|_\H+   \Big(1+ \tau^{\ff {1} 4}R^{m+1}  +\tau^{\frac 12}\vv^{-\frac 12}R^{2m+1}+ \tau^{\frac 12}  R^{\frac {4m+3} 2}\Big)^{\ff{1+2m/p}{m+1}}\Big).
  \end{split} \end{equation}
By \eqref{Z1'},  \eqref{ES1}, \eqref{PB} and \eqref{BN0},  we find constants $c_2,c_3>0$ such that
   \beg{align*} &\|1_{\{T_{R,\sigma}\ge t\}}u_t^{N,\tau}\|_{L^p(\OO; \dot\H^{\sigma})}\\
 &\le \|u_0\|_{\dot\H^{\sigma}}+c_2+c_2 \vv^{-1} \int_0^t  (t-s)^{-\ff \sigma 2}\e^{-\ff{\ll_1}2(t-s)} \Big(1+ \|u_{\tau_{k(s)}}^{N,\tau}\|_{L^{p}(\OO; L^{4m+2}(\mu))}^{2m+1}\Big) \d s
 \\
 &\le  \|u_0\|_{\dot\H^{\sigma}}+c_2+c_3\vv^{-1} \int_0^t  (t-s)^{-\ff \sigma 2}\e^{-\ff{\ll_1}2(t-s)} \Big(1+ \|u_{\tau_{k(s)}}^{N,\tau}\|_{L^{p}(\OO; \dot \H^{1})}^{2m+1}\Big) \d s.
   \end{align*}
This, together with \eqref{TT1} and \eqref{H1-pri},  yields that
  for some $c_4>0$
\begin{align*}
&\|1_{\{T_{R,\sigma}\ge t\}}u_t^{N,\tau}\|_{L^p(\OO; \dot\H^{\sigma})}\le c_4(1+\|u_0\|_{\dot\H^{\sigma}})+c_4\vv^{-\frac { (2m+3)}2} (1+\|u_0\|_{\dot\H^{1}})^{2m+1}.
\end{align*}
The proof of \eqref{LN0-4} is  complete.
\end{proof}

By Lemma \ref{LN0}(3)-(4) and a bootstrap argument, we have  the following results.

 \beg{lem}\label{LN1} Assume {\bf (A)}, {\bf (F)} and {\bf (W)} with $\aa_{m,d}:=\ff{md}{2m+1}<1,$ let $\eqref{Q}$ hold for some   $\aa\in [\aa_{m,d},1],$ and let $\sigma\in [\alpha_{m,d},1]$ and  $\gg_2$ be in $\eqref{DD}$.
 \beg{enumerate}  \item[$(1)$]  There exists a constant $c>0$   such that for $\tau$ satisfying $\eqref{TT}$ and $p= 6(2m+1)\gg_3$,
 \beq\label{LN1-1} \|u_t^{N,\tau}\|_{L^{p}(\OO;\dot \H^{\sigma})}\le  c(1+t) \vv^{-1} (1+\|u_0\|_{\dot\H^{\sigma}}).\end{equation}
 \item[$(2)$]   If $\beta \in (\sigma, 2)\cap  [\alpha,1+\alpha]$, then  there exists a constant $c>0$  such that for $\tau$ satisfying $\eqref{TT}$,
   \beq\label{LN1-2} \|u_t^{N,\tau}\|_{L^{6\gg_3}(\OO;\dot\H^\beta)}\le  c\vv^{-1}\big[(1+t)   \vv^{-1} (1+\|u_0\|_{\dot\H^{\beta}})\big]^{2m+1}.\end{equation}
 \end{enumerate}
 \end{lem}

 \beg{proof} (a)  Let
 $R= \vv^{\ff 3{2(4m+2)}} \tau^{-\ff {1}{(4m+2)\gg_2}} $. Then \eqref{X1} holds.
 Note that
 \beg{align*} & u_t^{N,\tau} = \pi_N u_0+ Z_t^N + \int_0^t \e^{-(t-s)A} \pi_N \Theta_{\tau,\sigma}(u_{\tau_{k(s)}}^{N,\tau})\d s,\\
 & 1_{\{t>T_{R,\sigma}\}}\le \sum_{i=0}^{k(t)} 1_{A_i},\ \ A_i:=\Big\{1_{\{i\tau\le T_{R,\sigma}\}} \sup_{s\in [i\tau,(i+1)\tau]}\|u_s^{N,\tau}\|_{\dot\H^{\sigma}}\ge R\Big\}.\end{align*}
 So, by \eqref{TA}, \eqref{Q}, \eqref{ES1} and \eqref{BN0},  we find a constant $c_1>0$ such that
 \beq\label{X2}\beg{split} & \|1_{\{t>T_{R,\aa} \}}u_t^{N,\tau}\|_{L^p(\OO; \dot\H^\sigma)}\le c_1 +\|u_0\|_{\dot\H^\sigma} +\int_0^t \big\|1_{\{t>T_{R,\sigma}\}} \e^{-A(t-s)} \Theta_{\tau,\sigma}(u_{\tau_{k(s)}}^{N,\tau})\big\|_{\dot\H^\sigma}\d s\\
 &\le c_1 +\|u_0\|_{\dot\H^\sigma} +(\tau\vv)^{-1} \int_0^t \e^{-\ff{\ll_1}2(t-s)} (t-s)^{-\ff\sigma 2} \P\big(t>T_{R,\sigma} \big)\d s\\
 &\le c_1 +\|u_0\|_{\dot\H^\sigma} +c_1 (\tau\vv)^{-1}  \sum_{i=0}^{k(t)}  \P(A_i)  \\
 &\le c_1 +\|u_0\|_{\dot\H^\sigma} +c_1 (\tau\vv)^{-1}  R^{-p} \sum_{i=0}^{k(t)} \E\Big[1_{\{i\tau\le T_{R,\sigma}\}} \sup_{s\in [i\tau,(i+1)\tau]}\|u_s^{N,\tau}\|_{\dot\H^\sigma }^p\Big].\end{split}\end{equation}
 By   \cite[Theorem 5.9]{Dap92}, \eqref{Q} implies that for some constant $c_2>0$,
 $$\bigg\|\sup_{s\in [i\tau,(i+1)\tau]} \bigg\|\int_{i\tau}^s \e^{-(s-r)A} \pi_N \d W_s \bigg\|_{\dot\H^\sigma}\bigg\|_{L^p(\OO)} \le c_2.$$
 Combining this with \eqref{X1}, \eqref{TA}, \eqref{ES1}, \eqref{BN0}, \eqref{PB} and $\sigma \ge \aa_{m,d},$ we find constants $c_3,c_4>0$ such that
 \beg{align*} &\bigg\|1_{\{i\tau\le T_{R,\sigma}\}} \sup_{s\in [i\tau,(i+1)\tau]} \|u_s^{N,\tau}\|_{\dot\H^\sigma}\bigg\|_{L^p(\OO)}
  \le c_2 + \big\|1_{\{i\tau\le T_{R,\sigma}\}}  u_{i\tau}^{N,\tau}\big\|_{L^p(\OO;\dot\H^\sigma)} \\
  &\quad + \bigg\| \sup_{s\in [i\tau,(i+1)\tau]} \int_{i\tau}^s1_{\{i\tau\le T_{R,\sigma}\}}\|\e^{-A(s-r)} \Theta_{\tau,\sigma}(u_{i\tau}^{N,\tau})\|_{\dot\H^\aa}\d r \bigg\|_{L^p(\OO)}\\
  &\le c_2 + c \vv^{-1} (1+\|u_0\|_{\dot\H^\sigma}) \\
  &\quad + \kk_1\vv^{-1} \bigg\| 1_{\{i\tau\le T_{R,\sigma}\}}\sup_{s\in [i\tau,(i+1)\tau]} \int_{i\tau}^s(s-r)^{-\ff \sigma 2} \e^{-\ff{\ll_1}2(s-r)} \Big(1+\|u_{i\tau}^{N,\tau}\|_{L^{4m+2}(\mu)}^{2m+1} \Big)  \d r \bigg\|_{L^p(\OO)}\\
  & \le  c_2 + c \vv^{-1} (1+\|u_0\|_{\dot\H^\sigma}) + c_3\vv^{-1} \Big(1+ \|1_{\{i\tau\le T_{R,\sigma}\}} u_{i\tau}^{N,\tau}\|_{L^{(2m+1)p}(\OO;\dot\H^\sigma)}^{2m+1}\Big)\\
  &\le    c_4 \vv^{-2m-2} (1+\|u_0\|_{\dot\H^\sigma})^{2m+1}. \end{align*}
  This together with  \eqref{X1}, \eqref{X2}  and $k(t)+1\le (t+1)\tau^{-1}$  leads to
  $$  \| u_t^{N,\tau}\|_{L^p(\OO; \dot\H^\sigma)} \le c_2  + c_0 \vv^{-\ff \sigma 2} (1+\|u_0\|_{\dot\H^\sigma})+ c_5  (1+t) \tau^{-2} R^{-p}   \vv^{-(2m+2)p-1} (1+\|u_0\|_{\dot\H^\sigma})^{(2m+1)p}$$
  for some $c_5>0.$
 By $R= \vv^{\frac 3{2(4m+2)}}\tau^{-\ff {1} {(4m+2)\gg_2}} $ and $p= 6(2m+1)\gg_2$, when \eqref{TT} holds we have
  $$ \tau^{-2} R^{-p}   \vv^{-(2m+2)p-1} (1+\|u_0\|_{\dot\H^\sigma})^{(2m+1)p} \le \vv^{-1} (1+\|u_0\|_{\dot\H^\sigma}),$$
  so that   \eqref{LN1-1} holds for some constant $c>0$.

 (b) Let $\beta \in (\sigma, 2)\cap  [\alpha,1+\alpha]$ and $p=6(2m+1)\gg_3$. By \eqref{ES1}, \eqref{PB}, \eqref{BN0},  \eqref{LN1-1} and $\sigma \ge \aa_{m,d},$   we find constants $c_1,c_2,c_3>0$ such that
 \beg{align*}  & \| u_t^{N,\tau}\|_{L^{6\gg_3}(\OO; \dot\H^\beta)}\\
 &\le \|u_0\|_{\dot\H^\beta}  +c_1+ c_1\vv^{-1}\int_0^t  (t-s)^{-\ff \beta 2}\e^{-\ff{\ll_1}2(t-s)} \Big(1+ \|u_{\tau_{k(s)}}^{N,\tau}\|_{L^{p}(\OO; L^{4m+2}(\mu))}^{2m+1}\Big) \d s\\
 &\le \|u_0\|_{\dot\H^\beta}  +c_1+ c_2\vv^{-1}\int_0^t (t-s)^{-\ff \beta 2}\e^{-\ff{\ll_1}2(t-s)} \Big(1+ \|u_{\tau_{k(s)}}^{N,\tau}\|_{L^{p}(\OO; \dot\H^{\aa_{m,d}})}^{2m+1}\Big) \d s\\
 &\le c_3 \vv^{-1}\big[(1+t)\vv^{-1} (1+\|u_0\|_{\dot\H^\beta})\big]^{2m+1}.\end{align*}
 So, \eqref{LN1-2} holds for $c=c_3.$
  \end{proof}

 Under the condition of Lemma \ref{LN1}, if $\tau$ satisfying $\eqref{TT}$ and $p= 6(2m+1)\gg_3$, we can use the same steps to verify that
  \beq\label{LN1-pri-l2} \|u_t^{N,\tau}\|_{L^{p}(\OO; \H)}\le  c(1+t) (1+\|u_0\|_{ \H}).\end{equation}

\begin{lem}\label{lem-cri} Assume  {\bf (A)}, {\bf (F)} and {\bf (W)} with $\aa_{m,d}:=\ff{md}{2m+1}=1,$  let $\eqref{Q}$ hold for   $\aa=1$,  let $\sigma\in (1,2)$ and $\gg_3$ be defined in \eqref{DD}.
\begin{enumerate}
\item[(i)]  There exists a constant $c>0$   such that for $\tau$ satisfying $\eqref{TT}$ and  $p= 6(2m+1)\gg_3$,
 \beq\label{LN1-1cri} \|u_t^{N,\tau}\|_{L^{p}(\OO;\dot \H^{\sigma})}\le  c(1+t) \vv^{-\ff {(2m+3)} 2} (1+\|u_0\|_{\dot\H^\sigma})^{2m+1} .\end{equation}
\item[(ii)]  For any $\beta \in (\sigma, 2)$,   there exists a constant $c>0$  such that for $\tau$ satisfying $\eqref{TT}$,
   \beq\label{LN1-2cri} \|u_t^{N,\tau}\|_{L^{6\gg_3}(\OO;\dot\H^\beta)}\le  c (1+t)^{2m+1}    \vv^{-\ff {(2m+1)(2m+3)} 2-1} (1+\|u_0\|_{\dot\H^\beta})^{(2m+1)^2} .\end{equation}
\end{enumerate}
\end{lem}

\begin{proof}
Let
 $R= \vv^{\ff 3{2(4m+2)}} \tau^{-\ff {1}{(4m+2)\gg_3}}.$ The proof is similar to that of Lemma \ref{LN1}.   The main difference is replacing \eqref{X1}  with \eqref{TT0} by \eqref{LN0-4} with \eqref{TT1}.
 Notice that  \eqref{X2} also holds. By \eqref{LN0-4}, there exists $c_1,c_2,c_3>0,$
 \beg{align*} &\bigg\|1_{\{i\tau\le T_{R,\sigma}\}} \sup_{s\in [i\tau,(i+1)\tau]} \|u_s^{N,\tau}\|_{\dot\H^\sigma}\bigg\|_{L^p(\OO)}
  \le c_1 + \big\|1_{\{i\tau\le T_{R,\sigma}\}}  u_{i\tau}^{N,\tau}\big\|_{L^p(\OO;\dot\H^\sigma)} \\
  &\quad + \bigg\| \sup_{s\in [i\tau,(i+1)\tau]} \int_{i\tau}^s1_{\{i\tau\le T_{R,\sigma}\}}\|\e^{-A(s-r)} \Theta_{\tau,\sigma}(u_{i\tau}^{N,\tau})\|_{\dot\H^\sigma}\d r \bigg\|_{L^p(\OO)}\\
  & \le  c_2  + c_2 \vv^{-\frac {2m+3}2} (1+\|u_0\|_{\dot\H^\sigma})^{2m+1} + c_2\vv^{-1} \Big(1+ \|1_{\{i\tau\le T_{R,\sigma}\}} u_{i\tau}^{N,\tau}\|_{L^{(2m+1)p}(\OO;\dot\H^{1})}^{2m+1}\Big)\\
  &\le    c_3 \vv^{-\frac {(2m+3)(2m+1)}2-1} (1+\|u_0\|_{\dot\H^\sigma})^{(2m+1)^2}. \end{align*}

  This together with  \eqref{LN0-4} , \eqref{X2}  and $k(t)+1\le (t+1)\tau^{-1}$  leads to
 \begin{align*}
  \| u_t^{N,\tau}\|_{L^p(\OO; \dot\H^\sigma)} &\le c_4 \vv^{-\ff {2m+3} 2} (1+\|u_0\|_{\dot\H^\sigma})^{2m+1}\\
  &+ c_4  (1+t) \tau^{-2} R^{-p}   \vv^{-\frac {(2m+1)(2m+3)}2p-p-1} (1+\|u_0\|_{\dot\H^\sigma})^{(2m+1)^2p}
  \end{align*}
  for some $c_4>0.$
 By $R= \vv^{\frac 3{2(4m+2)}}\tau^{-\ff {1} {(4m+2)\gg_3}} $ and $p= 6(2m+1)\gg_3$, when \eqref{TT} holds we have
  $$ \tau^{-2} R^{-p}   \vv^{-\frac {(2m+1)(2m+3)}2p-p-1} (1+\|u_0\|_{\dot\H^\sigma})^{(2m+1)^2p} \le \vv^{-\ff {2m+3} 2} (1+\|u_0\|_{\dot\H^\sigma})^{2m+1},$$
  so that   \eqref{LN1-1cri} holds for some constant $c>0$.

When $\beta\in (\sigma,2)$, the same procedures in the proof of (ii) in Lemma \ref{LN1} yields \eqref{LN1-2cri}.

\end{proof}

\section{Log-Harnack inequality  and gradient estimate}
\label{sec-4}
For $N\in \bar\N$, let $\B_b(\H_N)$ be the space of all bounded measurable functions on $\H_N$, let $\B_b^+(\H_N)$ be the class of positive bounded measurable functions on $\H_N$, and let
$C_{bl}(\H_N)$ be the set of all bounded Lipschitz functions on $\H_N$.  The Lipschitz constant of $\psi$ is defined as
 $$\|\nabla \psi\|_{\infty}:=\sup_{x\ne y}   \ff{|\psi(x)-\psi(y)|}{\|x-y\|_\H},$$
  which coincides with the uniform norm of the gradient $\nn\psi$ when $\psi\in C_b^1(\H_N)$.

Let $P_t^N$ be the Markov semigroup of $u_t^N$ on $\H_N,$
i.e. for any $\psi\in \B_b(\H_N),$
$$P_t^N \psi(x):=  \E \big[\psi (u_t^N(x))\big],\ \ t\ge 0, \ x\in \H_N,$$
where  $u_t^N(x)$ solves \eqref{sac} with initial value $u_0^N(x)=x\in \H_N$.

The following gradient estimate  is the key   to cancel the exponential term of $\vv^{-1}$ in the proof of Theorem \ref{T1}.

\beg{lem}\label{L3}  Assume  {\bf (A)}, {\bf (F)} and {\bf (W)}. For  any $N\in\bar \N$ and  $ \psi\in \B_{b}^+(\H_N)$,
\beq\label{*Y} P_t^N \log \psi(y)\le \log P_t^N \psi(x) +  \|\nn\log \psi\|_\infty \e^{-\ff1\vv t}\|x-y\|_\H +\ff {\gg(\vv)^2(1+\kk_1)^2}  {4 \vv}  \|x-y\|_\H^2,\ \ t>0, x,y\in\H_N.\end{equation}
Consequently, for any $ \psi\in C_{bl}(\H_N)$,
\beq\label{GG} \|\nn P_t^N \psi\|_\infty\le \e^{-\ff 1 \vv t}\|\nn\psi\|_\infty + \ff{\gg(\vv) (1 +\kk_1)}{2\ss{\vv  }}\|\psi\|_\infty,\ \  t\ge 0.\end{equation}
\end{lem}

\beg{proof} The proof is based on the coupling by change of measures, which is developed in \cite{W07} to establish Harnack type inequalities.

Observe that
\beq\label{O} \|u\|_{\dot \H_0^1}^2 \ge \ll_{N_\vv+1} \big\|(1-\pi_{N_\vv}) u\big\|_\H^2\ge \ff{1+\kk_1}{\vv} \|(1-\pi_{N_\vv})u\|_\H^2.\end{equation}
Let $u_t^N$ solve \eqref{sac} with $u_0^N=x,$ and consider
following coupled SDE with $\tt u_0^N=y$:
\beq\label{UN'} \d\tt u_t^N =  - \Big(A \tt u_t^N  + \ff 1 \vv \pi_N f(\tt u_t^N)+\ff{1+\kk_1}\vv \pi_{N_\vv}\big(\tt u_t^N-u_t^N\big)\Big) \d t+ \pi_N\d W_t.\end{equation}
By \eqref{F3}, \eqref{O}  and It\^o's formula,  we obtain
\beg{align*} &\d \|\tt u_t^N- u_t^N\|_\H^2\\
&\le -2 \bigg(\|\tt u_t^N- u_t^N\|_{\dot\H^1}^2 -\ff {\kappa_1}\vv \|\tt u_t^N- u_t^N\|_\H^2 +\ff{1+\kk_1}\vv \big\|\pi_{N_\vv} (\tt u_t^N- u_t^N)\big\|_\H^2\bigg)\d t\\
&\le -\ff{2 }\vv  \|\tt u_t^N- u_t^N\|_\H^2\d t.\end{align*}
Thus,
\beq\label{D1}  \|\tt u_t^N- u_t^N\|_\H \le \e^{-\ff1 \vv t}\|x-y\|_\H,\ \ t\ge 0.\end{equation}

On the other hand, let
$$\tt W_t:= W_t- \int_0^t Q^{\ff 1 2} \xi_s\d s,\ \ \ \  \xi_s:= \ff{1+\kk_1}\vv Q^{-\ff 1 2} \pi_{N_\vv} (\tt u_s^N-u_s^N).$$
Here the operator $Q^{\frac 12}$ is defined by $Q^{\ff 12}e_i:=\sqrt{q_i}e_i, i\in \bar{\mathbb  N},$ and $Q^{-\ff 12}$ is the pseudo-inverse of $Q^{\frac 12}$. 

By Girsanov's theorem, $(\tt W_s)_{s\in [0,t]}$ is a $Q$-Wiener process under the weighted probability $R_t\d\P$, where
$$R_t:= \e^{\int_0^t \<\xi_s, \d W_s\>_\H-\ff 1 2 \int_0^t \|\xi_s\|_\H^2\d s}.$$
Rewriting \eqref{UN'} as
$$\d\tt u_t^N =  - \Big(A \tt u_t^N(t) + \ff 1 \vv \pi_N f(\tt u_t^N) \Big) \d t+ \pi_N\d \tt W_t,$$
by the weak uniqueness of this SDE, we obtain that for any   positive function $\psi\in C_b^1(\H_N)$,
$$P_t^N \log \psi(y)= \E[R_t\log \psi(\tt u_t^N)]= \E [R_t\log \psi(u_t^N)] + \E\big[\log \psi(\tt u_t^N)-\log \psi(u_t^N)\big].$$
Combining this with Young's inequality (see \cite[Lemma 2.4]{ATW}) and \eqref{D1}, we obtain
\beq\label{*Y2} \beg{split} &P_t^N \log \psi(y)\le \log \E \psi(u_t^N) + \E[R_t\log R_t] + \|\nn\log \psi\|_\infty \e^{-\ff 1\vv t} \|x-y\|_\H\\
&= \log P_t^N \psi(x)+\ff 1 2 \E\bigg[R_t\int_0^t \|\xi_s\|_\H^2\d s \bigg]+  \|\nn\log \psi\|_\infty \e^{-\ff 1\vv t}  \|x-y\|_\H,\end{split}\end{equation}
where the last step follows from
$$\log R_t= \int_0^t \<\xi_s,\d \tt W_s\>_\H +\ff 1 2 \int_0^t \|\xi_s\|_\H^2\d s,$$
and that  $\int_0^t \<\xi_s,\d \tt W_s\>_\H$ is a martingale under the weighted probability $R_t\d\P$.
Noting that \eqref{AN} and \eqref{D1} imply
$$\|\xi_s\|_\H^2 \le \|Q^{-\ff 1 2}\|_{\mathcal L (\H_{N_\vv})}^2\Big(\ff{1+\kk_1}{\vv}\Big)^2  \e^{-\ff{1}\vv t}\|x-y\|_\H^2= \gg(\vv)^2 \Big(\ff{1+\kk_1}{\vv}\Big)^2\e^{-\ff{2 }\vv t}\|x-y\|_\H^2,$$
we deduce \eqref{*Y} from \eqref{*Y2}.   By \cite[Theorem 2.1]{BWY2019},   \eqref{GG} follows from \eqref{*Y}.

\end{proof}

Next, we estimate Sobolev norms of $\nn P_t^N$ which is helpful for improving the convergence order w.r.t $\tau$ and $\lambda_N^{-1}$ (see, e.g., \cite{Deb11,Kru14a} and Remark \ref{Remark 2.1.}). For a function $\psi\in C^1(\H_N)$ and a constant  $\aa \ge 0,$ let
\beg{align*} \|\nn \psi(x)\|_{\dot\H^{\aa}}:=\bigg(\sum_{i=1}^N \ll_i^{\aa}\big|\nn_{e_i} \psi(x)\big|^2 \bigg)^{\ff 1 2},\end{align*}
where $\nn_{e_i} \psi(x):=\<\nabla \psi(x),e_i\>.$
For any $x\in\H_N$, let $u_t^N(x)$ solve \eqref{sac} for   initial value   $u_0^N=x$.  We have the following result.

\beg{lem}\label{L4}  Assume  {\bf (A)}, {\bf (F)} and {\bf (W)}.  \beg{enumerate}
 \item[$(1)$] For any constants  $\aa_1\in (0,2)$ and $q\in [1,\infty)$ such that
\beq\label{00} \dd:=\ff{\aa_1} 2 +\ff{d }{4q}<1,\end{equation}
there exists  a constant $c>0$ such  that for any $t>0, N\in\bar\N, x\in\H_N$ and $\psi\in C_{bl}(\H_N),$
\beq\label{GR}\beg{split}&  \|\nn P_t^N\psi(x)\|_{\dot\H^{\aa_1}} \\
&\le c   \|\nn \psi\|_\infty\bigg[t^{-\ff{\aa_1} 2}
+ \e^{\ff{\kk_1t}\vv  }  \ll_1^{\dd-1}   \vv^{-1}  \sup_{s\in [0,t]}\E\|u_s^N(x)\|_{L^{4qm} (\mu)}^{2m}\bigg].\end{split} \end{equation}
\item[$(2)$]  Let $m(d-2)\le 1$ and let $\eqref{Q}$  hold for some  $\aa\in \big[\aa_{m,d},1\big]\cap\big(\ff{d-2}2,1\big]$. Then for any $\aa_1\in (0,2)$ and $\beta\in (\frac d2,1+\alpha]$,
there exists a constant $c>0$  such  that  any $t>0, N\in\bar\N, x\in\H_N$ and $\psi\in C_{b,l}(\H_N),$
\beq\label{GR2}   \beg{split} &\|\nn P_t^N\psi(x)\|_{\dot\H^{\aa_1}}  \\
&\le c   \|\nn \psi\|_\infty\bigg[t^{-\ff{\aa_1} 2}
+ \e^{\ff{\kk_1t}\vv  }    \vv^{-1}  \Big( \|x\|_{\dot\H^{\beta}}^{2m} + \vv^{-(md+2)m}\big(1+\|x\|_{\dot \H^1}^{2m(2m+1)}\big) \Big)\bigg].\end{split} \end{equation}
  Consequently,
\beq\label{GR3} \beg{split}
&\|\nn P_t^N\psi(x)\|_{\dot\H^{\aa_1}}   \le c  \Big(\e^{-\ff 1 \vv t}\|\nn\psi\|_\infty + \ff{\gg(\vv) (1 +\kk_1)}{2\ss{\vv  }}\|\psi\|_\infty\Big) \\
&\quad\times
 \bigg[ t^{-\ff{\aa_1} 2}
+      \vv^{-1}    \Big(\|x\|_{\dot\H^{\beta}}^{2m} + \vv^{-(md+2)m}\big(1+\|x\|_{\dot \H^1}^{2m(2m+1)}\big) \Big)\bigg].  \end{split}  \end{equation}
  \end{enumerate}
\end{lem}

\begin{proof} (1)   It is easy to see that for any $\eta_0\in \H_N$,
$$ \eta_t:=\lim_{\vv\to 0}\ff{u_t^N(x+\vv \eta_0)- u_t^N(x)}\vv$$
exists and solves the equation
\beq\label{00e}\eta_t' = -\pi_N\Big(A+\vv^{-1} f'(u_t^N(x))\Big)  \eta_t,\ \ t\ge 0.\end{equation}
By the chain rule,
\beq\label{01} \big|\<\nn P_t^N\psi(x), \eta_0\>_\H\big|= \big|\E \<\nn \psi(u_t^N(x)), \eta_t\>_\H\big|
\le \|\nn\psi\|_\infty \E \|\eta_t\|_\H. \end{equation}
  To estimate $\E \|\eta_t\|_\H$, we introduce the random operators
  $$\H_N\mapsto I_{s,t} h\in\H_N,\ \ t\ge s\ge 0,$$
  where for each $s\ge 0$ and $h\in \H_N$, $h_{s,t}:= I_{s,t} h$ solves the equation
\beq\label{02} \ff{\d}{\d t} h_{s,t}= -\pi_N\Big(A+\vv^{-1} f'(u_t^N(x))\Big)  h_{s,t},\ \ t\ge s,\ h_{s,s}=h.\end{equation}
By \eqref{00e},   $\tt \eta_t:= \eta_t-  \e^{-At}\eta_0$ solves the equation
$$\tt\eta_t'=  -\pi_N\Big(A+\vv^{-1} f'(u_t^N(x))\Big)  \tt\eta_t -\vv^{-1}\pi_N f'(u_t^N(x)) \e^{-At}\eta_0,\ \ t\ge 0,$$
so that
\beq\label{03} \eta_t= \tt\eta_t + \e^{-A t}\eta_0= \e^{-At}\eta_0- \ff 1\vv \int_0^t I_{s,t}
 \pi_N f'(u_t^N(x)) \e^{-As}\eta_0\d s.\end{equation}
By \eqref{02} and \eqref{F3}, we have
$$\ff{\d}{\d t} \|h_{s,t}\|_\H^2  \le \ff{2\kk_1}\vv \|h_{s,t}\|_\H^2,$$
so that
\beq\label{04} \|I_{s,t}h\|_\H=\|h_{s,t}\|_\H\le \e^{\ff{\kk_1}\vv (t-s)} \|h\|_\H,\ \ t\ge s\ge 0.\end{equation}
By \eqref{F4} and H\"older's inequality,  we find a constant $c_1>0$ such that for any $q\in  [1,\infty)$,
\beq\label{04'} \beg{split} &\big\|f'(u_s^N(x)) \e^{-(A+\ll)s}\eta_0\big\|_\H \le \|f'(u_s^N(x))\|_{L^{ 2q }(\mu)} \|\e^{-(A+\ll)s}\eta_0\|_{L^{\ff{2q}{q-1}}(\mu)}\\
&\le c_1   \|u_s^N(x)\|_{L^{ 4qm }(\mu)}^{2m}  \|\e^{- A s}\eta_0\|_{L^{\ff{2q}{q-1}}(\mu)}.\end{split}\end{equation}
 By  \eqref{Con-est} and \eqref{BN0}, we find   constants $c_2,c_3>0$,
 $$\|\e^{-As}\eta_0\|_{L^{\ff{2q}{q-1}}(\mu)}\le c_2 s^{-\ff{d}{4q}} \e^{-\ff{\ll_1}2 s} \|\e^{-\ff s 2 A }\eta_0\|_{L^2  (\mu)}\le  c_3 s^{-\dd} \e^{-\ff{\lambda_1}2 s} \|\eta_0\|_{\dot\H^{-\aa_1}}.$$
 Combining this with \eqref{00e}, \eqref{03},  \eqref{04} and \eqref{04'}, we find a constant $c_4>0$ such that
 $$ \E \|\eta_t\|_\H \le c_4 \|\eta_0\|_{\dot\H^{-\aa_1}}\bigg[t^{-\ff {\aa_1} 2}
+ \ff{\e^{\ff{\kk_1}\vv t}}\vv \int_0^t  s^{-\dd}\e^{-\ff{ \lambda_1}2 s} \E\|u_s^N(x)\|_{L^{4qm} (\mu)}^{2m} \d s\bigg].$$
Combining this  with   \eqref{01} and the fact that \eqref{00} implies
\beg{align*}&\int_0^t s^{-\dd} \e^{-  \ff {\lambda_1}{2}s }\d s\\
&=\int_0^{t\ll_1}   r^{-\dd}  \ll_1^{\dd-1} \e^{-r/2}\d r \le c_5 \ll_1^{\dd-1},\ \ t>0 \end{align*} for some constant $c_5>0$, we get the desired estimate \eqref{GR}.

(2)  Since $m(d-2)\le 1$ and  $\aa\in \big[\aa_{m,d},1\big]\cap\big(\ff{(d-2)^+}2,1\big]$,  by  \eqref{Sob1}, we can take $q=\infty.$   So, \eqref{GR} holds, hence for any $\aa_1\in (0,2),$
\beq\label{GR'} \beg{split}&  \|\nn P_t^N\psi(x)\|_{\dot\H^{\aa_1}} \\
&\le c   \|\nn \psi\|_\infty\bigg[t^{-\ff {\aa_1} 2}
+ \e^{\ff{\kk_1t}\vv  }   \vv^{-1}  \sup_{s\in [0,t]}\E\|u_s^N(x)\|_{L^{\infty} (\mu)}^{2m} \bigg].\end{split} \end{equation}
This together with \eqref{I1} and \eqref{I2},  implies that for any $\beta\in (\frac d2,1+\alpha],$
$$\sup_{s\ge 0} \E\|u_s^N(x)\|_{L^{\infty} (\mu)}^{2m}\le c_1 \|x\|_{\dot\H^{\beta}}^{2m} + c_1 \vv^{-(md+2)m} \big(1+\|x\|_{\dot \H^1}^{2m(2m+1)}\big)$$
for some constant $c_1>0$. Combining this with \eqref{GR'}, we prove \eqref{GR2} for some constant $c>0$.

Observe that  \eqref{GR2} implies \eqref{GR3} for $t\le 2\vv$.
When $t>2\vv$,  by $P_t^N\psi= P_\vv^N(P_{t-\vv}^N\psi)$, we deduce from \eqref{GR2} that for some constants $c_3 >0$
\beg{align*} &\|\nn P_t^N\psi(x)\|_{\dot\H^{\aa_1}} \le c_2 \|\nn P_{t-\vv}^N\psi\|_\infty \bigg[\vv^{-\ff {\aa_1} 2}
+      \vv^{-1}    \Big( \|x\|_{\dot\H^{\beta}}^{2m} + \vv^{-(md+2)m}\big(1+\|x\|_{\dot \H^1}^{2m(2m+1)}\big) \Big)\bigg]. \end{align*}
Combining this with \eqref{GG}, we derive \eqref{GR3} for $t>2\vv.$
\end{proof}

\section{Proof of Theorem \ref{T1}}
\label{sec-5}

Throughout this section, we simply denote
$$\ u_t=u_t(x),\ \ \ u_t^N=u_t^N(x),\ \ \ t\ge 0,\ N\in \N,$$
 {and define 
 $$I_\vv(x):= \|x\|_{\dot\H^{\aa_{m,d}} }+ \vv^{-\ff{md}2}  +\vv^{-\ff{md}2} \|x\|_{ \H }^{2m+1},\ \ x\in\dot\H^{\aa_{m,d}},\ \vv\in (0,1).$$}
\begin{proof}[Proof of Theorem \ref{T1}(1)]
   We split  the proof  into two steps by considering  $t\le\vv$ and $t>\vv$ respectively.

(a) By combining  \eqref{PB} with \eqref{F1} and \eqref{*2'},  for any constant $p\ge 1$,   we find constants
$c_1,c_2,c_3>0$ such that
\beq\label{KP0}\beg{split} & \sup_{N\in\bar \N}  \Big(\E\big[\|f(u_t^N)\|_\H^p\big] \Big)^{\ff 1 p}
\le c_1 \sup_{N\in\bar{\N}}\Big(\E\Big[1+\|u_t^N\|_{L^{4m+2}(\mu)}^{p(2m+1)}\Big]\Big)^{\ff 1 p}\\
&\le c_1 + c_2 \sup_{N\in\bar \N}\Big(\E\Big[\|u_t^N\|_{\dot\H^{\aa_{m,d}}}^{p(2m+1)}\Big]\Big)^{\ff 1 p}
\le c_3\Big[\|x\|_{\dot\H^{\aa_{m,d}}}^{2m+1} + \vv^{-\ff{\aa_{m,d}(2m+1)}2}\Big(1+ \|x\|_{\H}^{2m+1}\Big)\Big] \\
&=   c_3\Big(\|x\|_{\dot\H^{\aa_{m,d}}}^{2m+1} + \vv^{-\ff{md}2}  +\vv^{-\ff{md}2} \|x\|_{ \H }^{2m+1} \Big)=  c_3 I_\vv(x).\end{split}\end{equation}
Noting that
$$  u_t=  \e^{At} x +\vv^{-1}\int_0^t   \e^{A(t-s)}f(u_s)\d s + \sum_{i=1}^\infty \int_0^t q_i^{\ff 1 2} e_i \d W_i(s),$$
by \eqref{KP0}, for any $p\ge 1$, we find a constant $c_4>0$ such that   for any $N\in \N$,
\beq\label{KP1} \beg{split} &\Big(\E\big[\|u_t-\pi_N u_t\|_\H^p\big]\Big)^{\ff 1 p}
 \le   \big\|\e^{-At}(1-\pi_N)x\big\|_\H \\
 &\quad
+\ff 1 {\vv}  \int_0^t\e^{- \ll_{N+1}(t-s)} \Big(\E\big\|f(u_s)\big\|_\H^p\Big)^{\ff 1 p} \d s
+ c_4\bigg(\sum_{i=N+1}\int_0^t  \e^{-2\ll_i(t-s)}q_i\d s\bigg)^{\ff 1 2}\\
&\le  \e^{-\ll_{N+1}t} \|x-\pi_Nx\|_\H  +  c_3 \ll_{N+1}^{-1}\vv^{-1} I_\vv(x)   +c_4\ss{\dd_N}.  \end{split}\end{equation}
By \eqref{sac} and the definition of $\pi_N u_t$, we obtain
$$\d (u_t^N-\pi_N u_t)= -A(\pi_N u_t-u_t^N)\d t-\ff 1 \vv \pi_N \big[f(u_s)-f(u_t^N)\big]\d t.$$
So, by \eqref{F3} and \eqref{F4}, and noting that $\H= L^2(\mu)$, we derive
\beq\label{KP2} \beg{split} &\d\big\|u_t^N-\pi_N u_t\big\|_\H^2 = - \Big[2 \big\|u_t^N-\pi_N u_t\big\|_{\dot \H^1}^2+\ff 2 \vv\big\<f(u_t)-f(u_t^N), \pi_N u_t- u_t^N\big\>_\H\Big]\d t\\
&\le - \ff 2 \vv \Big[ \big\<f(u_t)-f(u_t^N), u_t- u_t^N\big\>_\H +\big\<f(u_t)-f(u_t^N),\pi_N u_t- u_t \big\>_\H \Big]\d t\\
&\le \ff{2\kk_1}\vv \Big[\big\|u_t- u_t^N\big\|_\H^2  +  \big\|f(u_t)-f(u_t^N)\big\|_\H \big\|\pi_N u_t- u_t\big\|_\H\Big]\d t.\end{split}\end{equation}
Combining this with \eqref{KP0} and \eqref{KP1} for $p=2$, we find a constant $c_5>0$ such that
\beg{align*}&\E\Big[\big|u_t^N-\pi_N u_t\big\|_\H^2\Big]\\
 &\le \ff{2 \kk_1}\vv \int_0^t \Big[\E \Big[\big\|u_s- u_s^N\big\|_\H^2\Big] + \Big(\E\big[2\|f(u_t)\|_\H^2+2\|f(u_s^N)\|_\H^2\big]\Big)^{\ff 1 2} \Big(\E\big[\|u_s-\pi_N u_s\|_\H^2\big]\Big)^{\ff 1 2}\Big]\d s\\
&\le \ff{c_5}{2\vv} \int_0^t  \E \Big[\big\|u_s- u_s^N\big\|_\H^2\Big] \d s +\ff{c_5}2\vv^{-1} \ll_{N+1}^{-1} I_\vv(x)   \|x-\pi_Nx\|_\H\\
 &\quad +\ff{c_5t}2     \Big( \vv^{-1}\ss{\dd_N} I_\vv(x)
  + \vv^{-2}\ll_{N+1}^{-1 }I_\vv(x)^2 \Big). \end{align*}
This together with \eqref{KP1} for $p=2$ yields that   for some constant $c_6>0$,
\beg{align*} & \E\Big[\big\|u_t^N-  u_t\big\|_\H^2\Big]\le 2 \E\Big[\big\|u_t^N-  \pi_Nu_t\big\|_\H^2+ \big\| \pi_Nu_t -  u_t\big\|_\H^2\Big]\\
&\le \ff{c_5}{2\vv} \int_0^t  \big\|u_s- u_s^N\big\|_\H^2 \d s +\ff{c_5}2 \vv^{-1}\ll_{N+1}^{-1} I_\vv(x)   \|x-\pi_Nx\|_\H\\
&\quad  +\ff{c_5}2 \ff t \vv    \Big( \ss{\dd_N} I_\vv(x)
  + \vv^{-1}\ll_{N+1}^{-1 }I_\vv(x)^2 \Big)
    + 6 \|x-\pi_Nx\|_\H^2 + 6c_3^2 \vv^{-2}\ll_{N+1}^{-2} I_\vv(x)^2    + 6 c_4^2 \dd_N\\
&\le  \ff{c_5}{2\vv} \int_0^t  \big\|u_s- u_s^N\big\|_\H^2 \d s+  c_6(t+1)\|x-\pi_N x\|_\H^2\\
 &\quad + c_6  (t+1)  \Big( \vv^{-1}\ss{\dd_N} I_\vv(x)  +\vv^{-2} \ll_{N+1}^{-1 }I_\vv(x)^2 \Big).
 \end{align*}
By Gronwall's lemma, this implies
\beq\label{KP3}\beg{split} &\E\Big[\big\|u_t^N-  u_t\big\|_\H^2\Big]\le  c_6\e^{\ff{c_5}{2\vv} t}(t+1) \Big[  \|x-\pi_N x\|_\H^2 +   \vv^{-1}\ss{\dd_N} I_\vv(x)  + \vv^{-2}\ll_{N+1}^{-1 }I_\vv(x)^2 \Big].  \end{split} \end{equation}
Since
$$\hat\W_1(u_t,u_t^N)\le \Big( \E\Big[\big\|u_t^N-  u_t\big\|_\H^2\Big]\Big)^{\ff 1 2},$$
\eqref{KP3} implies   \eqref{BJ1} for  $t\le  \vv$.

(b) Now, let $t> \vv$ be fixed. Choose $K_t\in \mathbb N$ such that $\ff t{K_t}\in (\vv,2\vv],$ so that
\beq\label{MM1} t_k:=\ff {kt}{K_t}\in (k\vv, 2k\vv],\ \ \ k\in {\mathbb N}.\end{equation}
For any integer $0\le k\le K_t-1$, and any $y\in \H_N$, let $u_{k,s}(y)$ and $u_{k,s}^N(y)$ solve the equations
\beq\label{KN}\beg{split} & \d u_{k,s}(y)= -\Big(A u_{k,s}(y) +\ff 1\vv f(u_{k,s}(y))\Big)\d s+ \d   W_{s+  t_k},\ \ u_{k,0}(y)=y,\\
& \d u_{k,s}^N(y)= -\Big(A u_{k,s}^N(y) +\ff 1\vv \pi_N f(u_{k,s}^N(y))\Big)\d s+\pi_N   \d   W_{s+  t_k},\ \ u_{k,0}^N(y)=y.\end{split}\end{equation}
 By the flow property for pathwise unique equations,  we have
 \beg{align*}&u_{k,t-t_k}(u_{t_k}^N(x))= u_{k+1,t-t_{k+1}} (u_{k,t_1}(u_{t_k}^N(x))),\\
  &  u_{k+1,t-t_{k+1}}  (u_{t_{k+1}}^N(x))= u_{k+1,t-t_{k+1}} (u_{k,t_1}^{N}(u_{t_k}^N(x))),\ \
  0\le k\le K_t-1,\end{align*}
  and by the definition of $P_t$
$$  \Big|\E\big[\psi(u_t(x))-\psi( u_t (\pi_Nx){)}\big]\Big|= \big|P_t \psi(x)- P_t\psi(\pi_Nx)\big|
\le \big\|\nn P_t \psi\big\|_\infty\big\|x-\pi_Nx\big\|_\H.$$
So,
\beq\label{M1} \beg{split} &  \Big| \E\big[\psi(u_t(x))-\psi(u_t^N(x))\big]\Big| -\big\|\nn P_t \psi\big\|_\infty\big\|x-\pi_Nx\big\|_\H\\
& \le  \Big|\E\big[\psi(u_t(x))-\psi(u_{K_t,0} (u_t^N(x)))\big]\Big|- \Big|\E\big[\psi(u_t(x))-\psi( u_t (\pi_Nx){\color{red})}\big]\Big|\\
&\le \E\big[\psi(u_{0,t}(\pi_N x))-\psi(u_{K_t,0} (u_t^N(x)))\big]\Big|\\
&\le  \sum_{k=0}^{K_t-1} \Big| \E\big[\psi(u_{k,t-t_k}(u_{t_k}^N(x)))-\psi(u_{k+1,t-t_{k+1}} (u_{t_{k+1}}^N(x)))\big]\Big|\\
&= \ \sum_{k=0}^{K_t-1}  \Big|\E\big[\psi(u_{k+1,t-t_{k+1}} (u_{k,t_1}(u_{t_k}^N(x))))-\psi(u_{k+1,t-t_{k+1}} (u_{t_{k+1}}^N(x)))\big]\Big|\\
&=   \sum_{k=0}^{K_t-1}  \big|P_{t-t_{k+1}} \psi(u_{k,t_1}(u_{t_k}^N(x)){)}-P_{t-t_{k+1}} \psi( u_{t_{k+1}}^N(x))\big| \\
&\le   \sum_{k=0}^{K_t-1} \|\nn P_{t-t_{k+1}}\psi\|_\infty \E\|u_{k,t_1}(u_{t_k}^N(x))- u_{k,t_1}^{N}(u_{t_k}^N(x)) \|_\H.\end{split}\end{equation}
By using $W_{\cdot+t_k}$ replacing   $W_\cdot$ in the proof of \eqref{KP3}, we derive the same estimate for $u_{k,t_1}-u_{k,t_1}^{N}$ in place of $u_{t}-u_t^N$ for $t=t_1  \le 2\vv$. Combining this with
$ u_{t_k}^N(x)-\pi_Nu_{t_k}^N(x)=0  $
 and $\aa_{m,d}=\ff{md}{2m+1}$, we obtain
\beg{align*} & \E\Big[\|u_{k,t_1}(u_{t_k}^N(x))- u_{k,t_1}^{N}(u_{t_k}^N(x)) \|_\H^2\Big]= \E  \Big[\E\|u_{k,t_1}(y)- u_{k,t_1}^{N}(y) \|_\H^2\Big]_{y=u_{t_k}^N(x)}\\
&\le   c_6\e^{c_5 }    (2\vv+1) \E \Big( \vv^{-1}\ss{\dd_N} I_\vv\big(u_{t_k}^N(x)\big)  +\vv^{-2} \ll_{N+1}^{-1 }I_\vv\big(u_{t_k}^N(x)\big)^2 \Big).\end{align*}
Since $m(d-2)\le 1$ implies $md\le  2m+1$, by \eqref{*2} and   \eqref{*2'}, we find a constant $c_7>0$ such that
\beg{align*}& \E\big[ I_\vv\big(u_{t_k}^N(x)\big)^2\big]\le 3 \E\big[\|u_{t_k}^N(x)\|_{\dot\H^{\aa_{m,d}}}^{4m+2} \big]+ 3 \vv^{-md}+3 \vv^{-md}\E\big[\|u_{t_k}^N(x)\|_{ \H }^{4m+2}\big]\\
&\le 3c^{4m+2} \big( \|x\|_{\dot\H^{\aa_{m,d}}}+ \vv^{-\ff {\aa_{m,d}} 2} +\vv^{-\ff {\aa_{m,d}} 2} \|x\|_\H\big)^{4m+2}+ 3 \vv^{-md}+\vv^{-md}\big(\|x\|_\H+ c\big)^{4m+2}\\
 &\le c_7 I_\vv(x)^2.\end{align*}
 Therefore, there exists a constant $c_8>0$ such that
\beg{align*}   \E\Big[\|u_{k,t_1}(u_{t_k}^N(x))- u_{k,t_1}^{N}(u_{t_k}^N(x)) \|_\H^2\Big]
 \le c_8  \Big( \vv^{-1}\ss{\dd_N} I_\vv (x)  +\vv^{-2} \ll_{N+1}^{-1 }I_\vv (x)^2 \Big).
\end{align*}
  Combining this with \eqref{GG}, \eqref{MM1}  and \eqref{M1}, we find a constant $c >0$ such that
\beg{align*} &\hat \W_1(u_t(x), u_t^N(x)):=\sup_{\|\psi\|_{b,1}\le 1} \Big| \E\big[\psi(u_t(x))-\psi(u_t^N(x))\big]\Big| \\
&\le c  \vv^{-\ff 1 2}\gg(\vv)  \Big[\|(1-\pi_N)x\|_\H+  \frac {t}{\vv}\Big( \vv^{-\ff 1 2 } \dd_N^{\ff 1 4}  I_\vv (x)^{\ff 1 2}   + \vv^{-1 }\ll_{N+1}^{-\ff 1 2 }I_\vv (x)   \Big)\Big].
\end{align*}
This implies \eqref{BJ1} for $t>\vv.$
 \end{proof}

\begin{proof}[Proof of Theorem \ref{T1}(2)-(3)] Denote
$$J_{\vv,p}(x) = \sup_{t\ge 0,N\in\bar\N} \E\Big(1+ \|u_t^N\|_{L^{2mp}(\OO;L^\infty(\mu))}^{2m}\Big) .$$
By Lemma \ref{L2'}, in the cases of    Theorem \ref{T1}  (2)-(3), we find  a constant $c_0>0$ such that
\beq\label{JJ}   J_{\vv,p}(x)
 \le c_0 \beg{cases} \|x\|_{\dot\H^\beta}^{2m} + \vv^{-\beta m} (1+\|x\|_\H^{2m}), &\text{case (2)},\\
\|x\|_{\dot\H^\bb}^{2m} +\vv^{-2m} \|x\|_{\dot\H^{\aa_{m,d}}}^{2m(2m+1)} + \vv^{-m (md+2)} (1+\|x\|_\H^{2m(2m+1)}), &\text{case (3)},
\end{cases} \end{equation}
where $\beta\in (\frac 12,1]$ for case (2) and $\beta\in (\frac d2,2)$ for case (3).
By \eqref{F4},
$$\|f(u_t)-f(u_t^N)\|_\H\le \kk_1\big(1+\|u_t\|_{L^\infty(\mu)}^{2m}+ \|u_t^N\|_{L^\infty(\mu)}^{2m}\big)\|u_t-u_t^N\|_\H.$$
So, \eqref{KP2} can be modified as
\beg{align*} &\d \|u_t^N-\pi_Nu_t\|_\H^2\\
&\le \ff{2\kk_1}{\vv} \Big[\|u_t-u_t^N\|_\H^2+ \big(1+\|u_t\|_{L^\infty(\mu)}^{2m}+ \|u_t^N\|_{L^\infty(\mu)}^{2m}\big)\|u_t-u_t^N\|_\H\|\pi_N u_t-u_t\|_\H\Big]\d t.\end{align*}
Combining this with \eqref{KP1} for $p=4$, we find   constants $c_1,c_2>0$ such that
\beg{align*} & \E\big[\|u_t^N-\pi_Nu_t\|_\H^2\big]\le  \ff{2\kk_1}{\vv} \int_0^t \E\big[\|u_s-u_s^N\|_\H^2\big]\d s\\
&\quad + \ff{c_1}\vv \int_0^t \big(\e^{-\ll_{N+1}s}\|x-\pi_N x\|_\H+ \ll_{N+1}^{-1} \vv^{-1} I_\vv(x)+\sqrt{\delta_N}\big) J_{\vv,4}(x) \big(\E[\|u_s-u_s^N\|_\H^2]\big)^{\ff 1 2}\d s\\
&\le  \ff{c_2}{\vv} \bigg(\int_0^t \E\big[\|u_s-u_s^N\|_\H^2\big]\d s + J_{\vv,4}(x)^2   \big(\ll_{N+1}^{-1} \|x-\pi_N x\|_\H^2 +  t\big(\ll_{N+1}^{-1}\vv^{-1}I_\vv(x) + \sqrt{\dd_N}\big)^2\bigg).\end{align*}
This together with \eqref{KP1} yields that that for some constant $c_3>0$,
\beg{align*} &\E\big[\|u_t^N- u_t\|_\H^2  \big] \le 2\E\big[\|u_t^N- \pi_Nu_t\|_\H^2  \big]+2\E\big[\|u_t- \pi_N u_t\|_\H^2  \big]\\
&\le
 \ff{2c_2}\vv \int_0^t \E\big[\|u_s-u_s^N\|_\H^2\big]\d s\\
 &\quad + c_3 J_{\vv,4}(x)^2 \Big[(\ll_{N+1}^{-1}\vv^{-1}+1\big)\|x-\pi_Nx\|_\H^2+(1+\frac t {\vv})\big(\ll_{N+1}^{-2}\vv^{-2} I_\vv(x)^2+ \dd_N\big)\Big].\end{align*}
By Gronwall's inequality, we derive that
\beg{align*} & \E\big[\|u_t^N-\pi_Nu_t\|_\H^2\big]\\
&\le c_3\e^{\ff{2c_2t}\vv} J_{\vv,4}(x)^2 \Big[(\ll_{N+1}^{-1}\vv^{-1}+1\big)\|x-\pi_Nx\|_\H^2+(1+\frac t{\vv})\big(\ll_{N+1}^{-2}\vv^{-2} I_\vv(x)^2+ \dd_N\big)\Big].\end{align*}
In particular,   we find a constant $c_4>0$ such that
\beq\label{VB} \beg{split} & \sup_{t\le\vv} \E\big[\|u_t^N-\pi_Nu_t\|_\H^2\big]\\
&\le c_4 \e^{\ff{2c_2t}\vv} J_{\vv,4}(x)^2 \Big[(\ll_{N+1}^{-1}\vv^{-1}+1\big)\|x-\pi_Nx\|_\H^2+ (1+\frac t{\vv})\big(\ll_{N+1}^{-2}\vv^{-2} I_\vv(x)^2+ \dd_N\big)\Big].\end{split} \end{equation}
Combining this with  \eqref{JJ}, we derive  \eqref{BJ2}-\eqref{BJ3} for $t\le\vv$. When $t>\vv$,
by repeating step (b) in the proof of Theorem \ref{T1}(1) with \eqref{VB} in place of \eqref{KP3}, and applying \eqref{JJ}, we finish the proof.
\end{proof}

\section{Proof of Theorem \ref{T2}  }
\label{sec-6}

Let $\beta \in (\aa_{m,d},1+\alpha]$, and let $\gg_1$ be given in \eqref{DD}. We have
\beq\label{M0} \beg{split} &q:= \ff d{d+4m\beta-2md}\in \Big(\frac 12, \frac {d}{(d-2\beta)^+}\Big),\ \ \ \ff{4mq}{q-1}=\ff{2d}{(d-2\beta)^+},\\
&\gg_1=\ff \beta 2 -\ff{d(q-1)}{4q}>0,\ \ \ \end{split}\end{equation}
 Below we prove \eqref{A}-\eqref{B} for $\aa_{m,d} \le \sigma \le 1$ with $\aa_{m,d}<1$, and \eqref{B'} for $\sigma>\aa_{m,d}=1$ respectively.

 \beg{proof}[Proof of Theorem \ref{T2} (1)] Let $\aa_{m,d}<1$, $\sigma\in [\aa_{m,d},1]$ and $\beta\in [\sigma,1].$

 (a)  For any $y\in \H_N$ and $i\in\mathbb Z,$ let $(u_{i,s}^N)_{s\in [0,\tau]}$ and $(u_{i,s}^{N,\tau})_{s\in [0,\tau]}$ solve the SDEs
 \beq\label{MX}\beg{split} & \d  u_{i,s}^N(y)= \pi_N\big(\Theta_0(u_{i,s}^N(y))-Au_{i,s}^N(y)\big)\d s + \d W_{i\tau+s},\ \ u_{i,0}^N(y)=y,\\
 &\d  u_{i,s}^{N,\tau}(y)= \pi_N\big(\Theta_{\tau,\sigma}(y)-Au_{i,s}^{N,\tau}(y)\big)\d s + \d   W_{i\tau+s},\ \ u_{i,0}^{N,\tau}(y)=y.\end{split}\end{equation}
 Then
  \beq\label{MX3} \beg{split} &\sup_{s\in [0,\tau]} \E\|u_{i,s}^N(y)-u_{i,s}^{N,\tau}(y)\|_\H =\sup_{s\in [0,\tau]}  \int_0^s \E\Big\|\e^{-A(s-r)} \big(\Theta_0(u_{i,r}^N(y))-\Theta_{\tau,\sigma}(y)\big)\Big\|_\H\d r \\
  &\le \tau \E\big[\|\Theta_0(u_{i,r}^N(y))-\Theta_0(y)\|_\H+\|\Theta_{\tau,\sigma}(y) - \Theta_0(y)\|_\H\big].\end{split}\end{equation}
Let $\scr L_{\xi}$ stand for the law of random variable $\xi.$ We have
\beq\label{MX0}  \L_{u_{i,s}^{N}(y)}=\L_{u_s^N(y)},\ \  \L_{u_{i,s}^{N,\tau}(y)}=\L_{u_s^{N,\tau}(y)},\ \ u_{i,s}^N(u_{i\tau}^N(x))= u_{i\tau+s}^N(x),\ \ u_{i,s}^{N,\tau}(u_{i\tau}^{{N,\tau}}(x))= u_{i\tau+s}^{N,\tau}(x).\end{equation}
By \eqref{GG} and \eqref{MX0},  there exists a constant $\kk>0$ such that for any $\psi$ with $\|\psi\|_{b,1}\le 1$ and $x\in \H_N,$
\beq\label{MX1}\beg{split} &\big|\E[\psi(u_t^N(x))-\psi(u_t^{N,\tau}(x){)}]\big|\\
&\le \E \sum_{i=0}^{{k(t)-1}} \Big|P_{t-(i+1)\tau}^N\psi(u_{i,\tau}^N(u_{i\tau}^{N,\tau}(x)))- P_{t-(i+1)\tau}^N\psi(u_{i,\tau}^{N,\tau}(u_{i\tau}^{N,\tau}(x)))\Big|\\
&\qquad  +\E  \Big|\psi\big(u_{k(t), t-k(t){ \tau}}^N(u_{k(t)\tau}^{N,\tau}(x))\big) - \psi\big(u_{k(t), t-k(t){\tau}}^{N,\tau} (u_{k(t)\tau}^{N,\tau}(x))\big) \Big|\\
 &\le \kk  \gg(\vv) \vv^{-\ff 1 2 }  \sum_{i=0}^{{k(t)-1}}  \E\|u_{i,\tau}^N(u_{i\tau}^{N,\tau}(x))-u_{i,\tau}^{N,\tau}(u_{i\tau}^{N,\tau}(x))\|_\H \\
 &\qquad + \E\|u_{k(t), t-k(t){\tau}}^N(u_{k(t)\tau}^{N,\tau}(x))- u_{k(t), t-k(t){\tau}}^{N,\tau} (u_{k(t)\tau}^{N,\tau}(x))\|_\H.\end{split}\end{equation}
{Here we use the convention that $\sum_{i=0}^{{k(t)-1}}(\cdot)=0$ if $k(t)=0.$}

 To estimate these expectations, we simply denote
 $$y_i= u_{i\tau}^{N,\tau}(x),\ \ \ 0\le i\le k(t).$$
  By \eqref{*2'}, \eqref{LN1-1}  for $p= 6(2m+1)\gg_2$,  \eqref{LN1-pri-l2}, and applying \eqref{MX0}, we find a constant $c_0>0$ such that
  \beq\label{MX2} \|y_i\|_{L^p(\OO, \dot\H^\beta)}\le c_0  \vv^{-1} a_{t,x},\ \ \
   \|u_{i,s}^N(y_i)\|_{L^p(\OO, \dot\H^\beta)}\le c_0  \vv^{-1} a_{t,x}, \end{equation}
   where
  \begin{align}\label{def-atx}
  a_{t,x}:=(1+t)(1+\|x\|_{\dot \H^{\beta}}).
  \end{align}
 Below, we use \eqref{MX2} to estimate   the  expectation in \eqref{MX3}.

 (b)
 Since $\Theta_0=-\vv^{-1} f$ and $\ff 1 2+ \ff {2m}{p}\le 1$ for $p=6(2m+1)\gg_2,$ by \eqref{F4}, H\"older's inequality, \eqref{M0} and \eqref{S}, we find a constant $c_1>0$ such that
 \beq\label{MX4}\beg{split} & \E\big[\|\Theta_0(u_{i,s}^N(y_i))-\Theta_0(y_i)\|_\H\big]\\
 &\le \kk_1\vv^{-1} \E\big\|(u_{i,s}^N(y_i)-y_i)(1+|y_i|^{2m}+|u_{i,s}^N(y_i)|^{2m})\big\|_{L^2(\mu)}\\
 &\le \kk_1\vv^{-1} \E\Big[\big\|u_{i,s}^N(y_i)-y_i\|_{L^{2q}(\mu)} \big(1+\|y_i\|_{L^{\ff{4mq}{q-1}}(\mu)}^{2m}+\|u_{i,s}^N(y_i)\|_{L^{\ff{4mq}{q-1}}(\mu)}^{2m}  \big)\Big]\\
 & \le c_1\vv^{-1} \E\Big[\big\|u_{i,s}^N(y_i)-y_i\|_{L^{2q}(\mu)} \big(1+\|y_i\|_{\dot\H^\beta}^{2m}+\|u_{i,s}^N(y_i)\|_{\dot\H^\beta}^{2m}  \big)\Big]\\
 & \le c_1\vv^{-1}  \big\|u_{i,s}^N(y_i)-y_i\|_{L^2(\OO;L^{2q}(\mu))} \Big(1+\|y_i\|_{L^{4m}(\OO;\dot\H^\beta)}^{2m}+\|u_{i,s}^N(y_i)\|_{L^{4m}(\OO;\dot\H^\beta)}^{2m}  \Big)\\
 &\le c_1  \Big(1+ c_0^{2m}\vv^{-2m} a_{t,x}^{2m}  \Big) \big\|u_{i,s}^N(y_i)-y_i\|_{L^2(\OO;L^{2q}(\mu))}.\end{split}\end{equation}
 By \eqref{M0} we have $2q\le \ff{2d}{(d-2\beta)^+}$, so that \eqref{S} and \eqref{Q} yield
 $$ I_1:=\bigg\|\int_0^s\e^{-(s-r)A} \d \widetilde W_{r+i\tau} \bigg\|_{L^2(\OO;L^{2q}(\mu))}\le c(\aa)  \bigg\|\int_0^s\e^{-(s-r)A} \d \widetilde W_{r+i\tau} \bigg\|_{L^2(\OO; \dot\H^\beta)}
 \le c_2 \ss s$$
 for some constant $c_2>0$.
 Next, since $p=6(2m+1)\gg_2\ge 2$ and
 $$\|A\e^{-\ff r 2 A} y_i\|_\H\le c  r^{\ff\beta 2 -1}\|y_i\|_{\dot\H^\beta}$$
 holds for some constant $c >0$, by \eqref{Con-est}  we find   constants $c_3,c_4>0$ such that
\beg{align*} I_2&:=\big\|(\e^{-s A}-1)y_i\big\|_{L^2(\OO;L^{2q}(\mu))}=\bigg\|\int_0^s \e^{-\ff r 2 A} A\e^{-\ff r 2 A} y_i\d r\bigg\|_{L^2(\OO;L^{2q}(\mu))}\\
& \le c_3 \int_0^s \e^{-\ff{\ll_1}2 r} r^{\ff\beta 2 -1-\ff{d(q-1)}{4q}}  \| y_i\|_{L^2(\OO;\dot\H^\beta)} \d r \le c_4  \tau^{\gg_1} \| y_i\|_{L^p(\OO;\dot\H^\beta)},\ \ s\in [0,\tau],\end{align*}
where $\gamma_1$   defined in \eqref{DD} satisfies \eqref{M0}.
  Moreover,    by \eqref{F1},   \eqref{Con-est}, \eqref{M0}  and $p=6(2m+1)\gg_2\ge 2,$  we find constants $c_5,c_6 >0$ such that
 \beg{align*} &I_3:= \bigg\|\int_0^s\e^{-(s-r)A} \Theta_0(u_{i,r}(y_i))\d r\bigg\|_{L^2(\OO;L^{2q}(\mu))}\le \kk_1 \vv^{-1} \int_0^s \|\e^{-(s-r)A} f(u_{i,r}^N(y_i))\|_{L^2(\OO;L^{2q}(\mu))}\d r \\
 &\le c_ 5  \vv^{-1} \int_0^s \e^{-\ff{\ll_1}2 (s-r)} (s-r)^{-\ff{d(q-1)}{4q}} \big\|1+ |u_{i,r}^N(y_i)|^{2m+1}\|_{L^2(\OO; L^2(\mu))}   \d r \\
 &\le c_5  \vv^{-1} \int_0^s \e^{-\ff{\ll_1}2 (s-r)} (s-r)^{-\ff{d(q-1)}{4q}} \big(1+\|u_{i,r}^N(y_i)\|_{L^2(\OO;L^{4m+2}(\mu))}^{2m+1} \big) \d r\\
 &\le c_6 \vv^{-1} \int_0^s \e^{-\ff{\ll_1}2 (s-r)} (s-r)^{-\ff{d(q-1)}{4q}} \big(1+\|u_{i,r}^N(y_i)\|_{L^p(\OO;\dot\H^{\aa_{m,d}})}^{2m+1}  \big)\d r\\
 &\le c_6 \vv^{-1}\tau^{\gg_1+\ff{2-\aa_{m,d}}2} \sup_{s\in [0,\tau]} \big(1+c(\aa_{m,d})^{2m+1}\|u_{i,s}^N(y_i)\|_{L^p(\OO;\dot\H^{\aa_{m,d}})}^{2m+1} \big),\ \ s\in [0,\tau].\end{align*}
Combining the above estimates with
  \eqref{MX} and   \eqref{MX2},  we find a constant $c_7>0$ such that
 \beq\label{MX*}  \beg{split} &\sup_{s\in [0,\tau]}\|u_{i,s}^N(y_i)-y_i\|_{L^2(\OO; L^{2q}(\mu)} \le I_1+I_2+I_3\\
&\le  c_7\Big[\tau^{\ff 1 2}+ \tau^{\gg_1}\vv^{-1} a_{t,x}+  \tau^{\gg_1+\ff{2-\aa_{m,d}}2} \vv^{-2m-2}  a_{t,x}^{2m+1}\Big].
\end{split} \end{equation}
This together with \eqref{MX4} and   $k(t)+1\le (1+t)\tau^{-1}$ yields that for some constant $c_8>0$
\beq\label{MX5}  \beg{split} & \sum_{i=0}^{k(t)} \E\big[\|\Theta_0(u_{i,r}^N(y_i))-\Theta_0(y_i)\|_\H\big] \\
& \le c_8  (1+t) \tau^{-1}\Big[\tau^{\ff 1 2} \vv^{-2m} a_{t,x}^{2m}
   + \tau^{\gg_1} \vv^{-2m-1} a_{t,x}^{2m+1}\\
   &\qquad +   \tau^{\gg_1+\ff{2-\aa_{m,d}}2} \vv^{ -4m-2} a_{t,x}^{4m+1}\Big].\end{split}\end{equation}

(c) By \eqref{TA'}, \eqref{S}, \eqref{MX2},
\begin{align}\label{def-gamma2}
4m+2\le \ff{2d}{(d-2\beta)^+},\ \ \  4m+2\le 6(2m+1)\gg_2,
\end{align}
 we find a constant $c_9>0$ such that
\beq\label{MX6} \beg{split} &\E\big[\|\Theta_0(y_i)-\Theta_{\tau,\sigma}(y_i)\|_\H\big]\\
 &\le c\tau \vv^{-1} \E\big[  (1+\|y_i\|_{\dot \H^{\alpha_{m,d}}}^{2m+1}) \|y_i\|^{(2m+1)}_{\dot \H^{\sigma}}\big]  \le c_9  \tau \vv^{-1}(\vv^{-1}a_{t,x})^{4m+2}.\end{split}\end{equation}
Combining this with \eqref{MX1}, \eqref{MX3} and \eqref{MX5},  we find a constant $c'>0$ such that
   \beg{align*}  &\hat\W_1(u_t^N(x), u_t^{N,\tau}(x))\le  c'(1+t) \gg(\vv)  \tau\vv^{-\ff 32}(\vv^{-1}a_{t,x})^{4m+2}\\
  &  + c'   (1+t)\gg(\vv) \vv^{-2m-\ff 1 2} \big(\tau^{\ff 1 2}  a_{t,x}^{2m} +
\tau^{\gg_1} \vv^{-1} a_{t,x}^{2m+1}+ \tau^{\gg_1+\ff{2-\aa_{m,d}}2} \vv^{-2m-2}a_{t,x}^{4m+1}\big).  \end{align*}
Combining this with the fact that the condition \eqref{TT} and $\alpha_{m,d}\le \beta\le 1$  ensure
$$\tau^{\gg_1+\ff{2-\aa_{m,d}}2} \vv^{-2m-2}a_{t,x}^{4m+1}\le \tau^{\gg_1} a_{t,x}^{4m+1},$$ we obtain \eqref{A}.
 \end{proof}

\beg{proof}[Proof of Theorem \ref{T2}(2)]  Let $\aa_{m,d}<1$, $\sigma\in [\aa_{m,d},1]$ and $\beta\in (\sigma,1+\aa].$
  We will    use notations introduced in   the proof of  Theorem \ref{T2}(1).

By \eqref{*2''}, \eqref{LN1-2}, and \eqref{LN1-pri-l2}, for   $\beta \in (\sigma,1+\alpha],$ instead of \eqref{MX2} we have
 $$ \|y_i\|_{L^{6\gg_3}(\OO, \dot\H^\beta)}\le c_0   \vv^{-2m-2} a_{t,x}^{2m+1},\ \ \
   \|u_{i,s}^N(y_i)\|_{L^{6\gg_3}(\OO, \dot\H^\beta)}\le c_0   \vv^{ -2m-2}a_{t,x}^{2m+1}.$$
Using these estimates in place of \eqref{MX2}, so that \eqref{MX4}, \eqref{MX*},   \eqref{MX5} and \eqref{MX6}  can be modified as  follows respectively:
 \beg{align*} & \E\big[\|\Theta_0(u_{i,s}^N(y_i))-\Theta_0(y_i)\|_\H\big]\\
 &\le c_1  \Big(1+ c_0^{2m}\vv^{-2m(2m+2)} a_{t,x}^{2m(2m+1)}  \Big) \big\|u_{i,s}^N(y_i)-y_i\big\|_{L^2(\OO;L^{2q}(\mu))},\end{align*}
    \beg{align*} &\sup_{s\in [0,\tau]}\|u_{i,s}^N(y_i)-y_i\|_{L^2(\OO; L^{2q}(\mu)}  \\
&\le  c_7\Big[\tau^{\ff 1 2}+ \tau^{\gg_1}\vv^{-2m-2} a_{t,x}^{2m+1}+  \tau^{\gg_1+\ff{2-\aa_{m,d}}2}  \vv^{-2m-2}a_{t,x}^{2m+1}\Big],
\end{align*}
  \beg{align*} &\sum_{i=0}^{k(t)} \E\big[\|\Theta_0(u_{i,r}^N(y_i))-\Theta_0(y_i)\|_\H\big]  \\
  &\le c_8  (1+t)\tau^{-1}  \vv^{-2m(2m+2)} a_{t,x}^{2m(2m+1)}  \big[\tau^{\ff 1 2}     +   \tau^{\gg_1} \vv^{-2m-2} a_{t,x}^{2m+1} \\
     &\qquad + \tau^{\gg_1+\ff{2-\aa_{m,d}}2} \vv^{-2m-2}a_{t,x}^{2m+1} \big],\end{align*}
   and
  \beg{align*} &\E\big[\|\Theta_0(y_i)-\Theta_{\tau,\sigma}(y_i)\|_\H\big] \le c_9  \tau \vv^{-(2m+2)^2}a_{t,x}^{(2m+1)(2m+2)}.\end{align*}
These together with \eqref{MX1} imply the desired estimate \eqref{B}.
\end{proof}

 \beg{proof}[Proof of Theorem \ref{T2} (3)] Let $\aa_{m,d}=1, \si\in (1,2)$ and   $\beta\in [\sigma,2).$
By \eqref{*2''},  \eqref{LN1-1cri}-\eqref{LN1-2cri},  instead of \eqref{MX2} we have
  \begin{align*}
   &\|y_i\|_{L^{6\gg_3}(\OO, \dot\H^\sigma)}+   \|u_{i,s}^N(y_i)\|_{L^{6\gg_3}(\OO, \dot\H^\sigma)}\le c_0   \vv^{-\frac {2m+3}2} a_{t,x}^{2m+1},\\
   &\|y_i\|_{L^{6\gg_3}(\OO, \dot\H^\beta)}+   \|u_{i,s}^N(y_i)\|_{L^{6\gg_3}(\OO, \dot\H^\beta)}\le c_0   \vv^{-\frac {(2m+1)(2m+3)}2-1} a_{t,x}^{(2m+1)^2}.
   \end{align*}
By using these estimates in place of \eqref{MX2}, we may modify  \eqref{MX4}, \eqref{MX*},   \eqref{MX5} and \eqref{MX6}    as  follows respectively:
 \beg{align*} & \E\big[\|\Theta_0(u_{i,s}^N(y_i))-\Theta_0(y_i)\|_\H\big]\\
 &\le c_1  \Big(1+ c_0^{2m}\vv^{-m(2m+1)(2m+3)-2m} a_{t,x}^{2m(2m+1)^2}  \Big) \big\|u_{i,s}^N(y_i)-y_i\big\|_{L^2(\OO;L^{2q}(\mu))},\end{align*}
    \beg{align*} &\sup_{s\in [0,\tau]}\|u_{i,s}^N(y_i)-y_i\|_{L^2(\OO; L^{2q}(\mu)}  \\
&\le  c_7\Big[\tau^{\ff 1 2}+ \tau^{\gg_1} \vv^{-\frac {(2m+1)(2m+3)}2-1} a_{t,x}^{(2m+1)^2}+  \tau^{\gg_1+\ff{2-\aa_{m,d}}2} \vv^{-1}(\vv^{-\frac {2m+3}2} a_{t,x}^{2m+1})^{2m+1} \Big],
\end{align*}
  \beg{align*} &\sum_{i=0}^{k(t)} \E\big[\|\Theta_0(u_{i,r}^N(y_i))-\Theta_0(y_i)\|_\H\big]  \\
  &\le c_8  (1+t)\tau^{-1}  \vv^{-m(2m+1)(2m+3)-2m} a_{t,x}^{2m(2m+1)^2}  \big[\tau^{\ff 1 2}     +   \tau^{\gg_1} \vv^{-\frac {(2m+1)(2m+3)}2-1} a_{t,x}^{(2m+1)^2} \\
     &\qquad + \tau^{\gg_1+\ff 12} \vv^{-\frac {(2m+3)(2m+1)}2-1}a_{t,x}^{(2m+1)^2} \big],\end{align*}
   and
  \beg{align*} &\E\big[\|\Theta_0(y_i)-\Theta_{\tau,\sigma}(y_i)\|_\H\big] \le c_9  \tau \vv^{-(2m+1)(2m+3)-1}a_{t,x}^{2(2m+1)^2}.\end{align*}
Then the  estimate   \eqref{B'} follows from \eqref{MX1}.
 \end{proof}

\section*{Appendix}\label{appendix} 

   In this appendix, we discuss how to improve the weak convergence order of the spectral Galerkin method such that its convergence error is still depending $\vv^{-1}$ polynomially. First, we give the second order estimate of the Markov semigroup $P_t^N$.
   Then we establish the sharper convergence rate of the spectral Galerkin method via the Kolmogorov equation with some additional assumptions on $f$.

\subsection{Improving the convergence order via Kolmogorov equation}

 For a function $\psi\in C^2(\H_N)$ and constants $\aa_1,\aa_2\ge 0,$ let
\beg{align*}\|\nn^2 \psi(x)\|_{\dot\H^{\aa_1, \aa_2}}:=\bigg(\sum_{i,j=1}^N \ll_i^{\aa_1} \ll_j^{\aa_2} \big|\nn^2_{e_i,e_j} \psi(x) \big|^2  \bigg)^{\ff 1 2},\ \ x\in\H_N.\end{align*}
Denote $C_b^2(\H_N)$ the space of functions in $ C^2(\H_N)$ with bounded first and second derivates.

 Notice that for $x, h_1,h_2\in \H_N$, it holds that
 \begin{align}\label{rep-kol}
\nabla^2P_t^N \psi(x) \cdot (h_1,h_2)&=\E^{x} \left[\nabla \psi(u^{N}_t)\cdot \zeta_N^{h_1,h_2}(t)\right]\\\nonumber
 &\quad+\E^{x}\left[\nabla^2\psi({u^N_t})\cdot (\eta_N^{h_1}(t),\eta_N^{h_2}(t))\right].
 \end{align}
 Here we use $\E^{x}$ to indicate the dependence of the expectation on $x$.
 In addition assume that
\begin{align}\label{F5}
|f''(s)|  \le  \kk_1(1+|s|^{\max(2m-1,0)}).
\end{align}
Then $\zeta_N^{h_1,h_2}$ satisfies the second order variational equation
 \begin{align}\label{var-2}
  \partial_t \zeta_N^{h_1,h_2}&=-A\zeta_N ^{h_1,h_2}-\vv^{-1}\pi_N f'(u^N_t) \cdot \zeta_N^{h_1,h_2}\\\nonumber
  &
  -\vv^{-1}\pi_N f''(u^N_t)\cdot \left(\eta_{N}^{h_1},\eta_{N}^{h_2}\right),\quad \; \zeta_N^{h_1,h_2}(0)=0,
 \end{align}
respectively.

\begin{lem}\label{second-finite}Assume {\bf (A)}, {\bf (F)} and {\bf (W)} with $  \aa_{m,d}:= \ff{md}{2m+1}\le 1$,  let $\eqref{Q}$ hold  for some  $\aa\in \big[\aa_{m,d},1\big]\cap\big(\ff{d-2}2,1\big]$,  and  let  $\gamma(\vv)<\infty$ in $\eqref{AN}$. Let \eqref{F5} hold, $N\in \bar \N$ and $d\in (0,4).$
Then
 $\|\nn^2 P_t^N \psi(x)\|_{\dot\H^{0,0}}$ is finite if $ \psi\in \mathcal C_b^2(\H_N)$ and $x\in \dot \H^{\beta} \cap \H_N$ with $\beta\in (\frac d2,1+\alpha].$
\end{lem}

Lemma \ref{second-finite} only shows the finiteness of the second order derivative of $P_t^N$. From its proof in next subsection, one can see that the upper bound of $\|\nn^2 P_t^N \psi(x)\|_{\dot\H^{0,0}}$ may depend on $\vv^{-1}$ exponentially.

Next, we impose another additional condition \eqref{add-condition1} such that convergence order in space of the spectral Galerkin method can be improved.  A typical example of $f$ and $A$ satisfying this  condition is Example \ref{ex-f} with $m\in \mathbb N$ (one can use the algebra property of $\dot \H^{\beta}, \beta>\frac d2,$ to verify this). 

\begin{prp}\label{prp-high-order} Under the condition of Lemma \ref{second-finite}. Let $N\ge N_{\vv}$ where $N_{\vv}$ is  defined in \eqref{AN}.
In addition assume that
\begin{align}\label{add-condition1}
\|f(x)\|_{\dot \H^{\beta}}\le c(1+\|x\|_{\dot \H^{\beta}}^{2m+1}), \; \text{for any} \; \beta \in \big(\frac d2,2\big) \; \text{and}\; x\in \dot \H^{\beta}.
\end{align}
For any $\alpha_1\in (0,2)$ and $\beta\in (\frac d2,1+\alpha]$, there exists $c>0$ such that for any $t>0$,  $x\in \dot \H^{\beta}$  and $\psi\in C_b^2( \H)$, 
\begin{align*}
&|\E [{\psi}(u_t(x))]-\E[{\psi}(u_t^N(\pi_N x))]|\\
&\le c (1+t+t^{-\frac {\alpha_1} 2})\Big(\e^{-\ff 1 \vv t}\|\nn\psi\|_\infty + \ff{\gg(\vv) }{\ss{\vv  }}\|\psi\|_\infty\Big)  \Big(\lambda_N^{-\frac {\alpha_1} 2-\frac {\beta} 2}+\sqrt{\delta_N} +\|(1-\pi_N)x\|_{{\dot \H}^{-\alpha_1}}\Big)\\
&\quad \times  \vv^{-  {(md+2)}(2m+\frac 12)-2-m(2m+1)}  \Big(1+\|x\|_{\dot \H^{\beta}}^{4m+1}+\|x\|_{\dot \H^1}^{(2m+1)(4m+1)}\Big).
\end{align*}
\end{prp}

\begin{proof}
First, we define the following SDE with a truncated noise of parameter $M\in \mathbb N$,
\begin{align*}
\d U^{(M)}=-A U^{(M)} \d t-\vv^{-1}f(U^{(M)}) \d t+\pi_M \d W(t), U^{(M)}(0)=x.
\end{align*}
By the triangle inequality, for any ${\psi}\in C_b^2(\H),$ it holds that
\begin{align*}
&|\E [{\psi}(u_t(x))]-\E[{\psi}(u_t^N(\pi_N x))]|\\
&\le |\E [{\psi}(u_t(x))]-\E[{\psi}(U^{(M)}(t,x))]|+|\E [{\psi}(U^{(M)}(t,x))]-\E[{\psi}(u_t^N(\pi_N x))]|.
\end{align*}
Let  $M\ge N_{\vv}$ in \eqref{AN}. It suffices to estimate the above two terms.

\textbf{(a) Upper bound of $|\E [{\psi}(u_t(x))]-\E[{\psi}(U^{(M)}(t,x))]|$.}

By It\^o's formula, we obtain that
\begin{align*}
\E [\|U^{(M)}(t)-u_t\|^2_{\H}]\le 2 \kappa_1 \vv^{-1} \int_0^t \E [\|U^{(M)}(s)-u_s\|^2_{\H}] \d s+\delta_M t.
\end{align*}
The Gronwall's argument leads to
\begin{align*}
\E [\|U^{(M)}(t)-u_t\|^2_{\H}]\le c_1  e^{ \kappa_1 \vv^{-1} t} t \delta_M.
\end{align*}
As in the proof of (1) in Theorem \ref{T1}, it is not hard to see that for any $x\in \H,$
\begin{align}\label{trun-err}
|\E [{\psi}(u_t(x))]-\E[{\psi}(U^{(M)}(t, x))]|\le c_2 (1+t)\ff{\gg(\vv)}{\vv}  \sqrt{\delta_M}.
\end{align}

\textbf{(b) Upper bound of $|\E [{\psi}(U^{(M)}(t, x))]-\E[{\psi}(u_t^N(\pi_N x))]|$.}

We decompose $|\E [{\psi}(U^{(M)}(t, x))]-\E[{\psi}(u_t^N(\pi_N x))]|$  via the Kolmogorov equation instead of the telescope summation used in the proof of Theorem \ref{T1}.
Let $M=N \ge N_{\vv}.$
 Denote  $X^{(N)}(t,y)=P_{t}^{(N)}\psi(y)=\E \psi(U^{(N)}(t,y))$ for $\psi\in  C_b^2(\H)$ and $t\ge 0.$ Then it can be verified (see, e.g., \cite{BG18B,CH18}) that $X^{(N)}$ satisfies
\begin{align} \label{kol-eq}
\partial_t X^{(N)} (t,y)&=\nn X^{(N)} (t,y)\cdot \left[-Ay-\vv^{-1} f(y)\right]\\\nonumber
& +\frac 12 \sum_{j=1}^N  \nn^2 X^{(N)}(t,y)\cdot \left(\pi_N Q^{\frac 12}e_j,\pi_N Q^{\frac 12}e_j\right).
\end{align}
Now, we use the following decomposition,
\begin{align}\label{weak-dec}
& \E\big[\psi(U^{(N)}(t,x))-\psi(u_t^N(\pi_N x))\big]  \\\nonumber
&=\E \big[\psi(U^{(N)}(t,x))-\psi(U^{(N)}(t,\pi_Nx))\big]+\E[\psi(U^{(N)}(t,\pi_N x))-\psi(u_t^N(\pi_N x))] \\\nonumber
&=Er_1+Er_2.
\end{align}
For the term $Er_1,$ by using \eqref{I2} and \eqref{GR3},  we have that for any $ \alpha_1\in (0,2)$ and $\beta\in (\frac d2, 1+\alpha],$ there exists a constant $c_1>0$ such that
\begin{align*}
|Er_1|&\le c_1 \Big(\e^{-\ff 1 \vv t}\|\nn\psi\|_\infty + \ff{\gg(\vv) (1 +\kk_1)}{2\ss{\vv  }}\|\psi\|_\infty\Big) \Big[t^{-\ff {\aa_1} 2}
+  \vv^{-1} \big(1+\|x\|_{\dot\H^{{\beta}{}}}^{2m}\big)\\\nonumber
&+\vv^{- ({md+2})m-1}\Big(1+\|x\|_{\dot\H^1}^{(2m+1)2m}\Big) \Big] \|(1-\pi_N)x\|_{\dot \H^{-\alpha_1}}.
\end{align*}
For the term $Er_2$, applying the It\^o's formula and \eqref{kol-eq},  thanks to Lemma \ref{second-finite}, we have that
\begin{align*}
&|Er_2|=\Big|\int_{0}^{t} \frac {\text {d}}{\text{d} s} \E [X^{(N)}(t-s, u_s^N(\pi_N x))] \text{d} s\Big|\\
&=\Big|\int_0^t \vv^{-1} \E\Big[\<\nabla X^{(N)}(t-s, u^N_s(\pi_N x)), f(u^N_s( \pi_N x))- \pi_N f(u^N_s( \pi_N x))\>_\H\Big] \text{d} s\\
	&+\frac 12 \int_0^t \sum_{j=1}^{N}\E\Big[ \nabla^2X^{(N)}(t-s, u^N_s(\pi_N x))\cdot \{-(Q^{\frac 12}e_j, Q^{\frac 12}e_j) +(\pi_N Q^{\frac 12}e_j, \pi_NQ^{\frac 12}e_j)\}\Big] \text{d} s\Big|\\
	&=\vv^{-1}\Big|\int_0^t  \E\Big[\<\nabla X^{(N)}(t-s, u^N_s(\pi_N x)), \pi_N f(u^N_s( \pi_N x))-f(u^N_s( \pi_N x))\>_\H \Big] \text{d} s.
\end{align*}
By \eqref{GR3}, we obtain that for any $\alpha_1\in (0,2)$ and $\beta\in (\frac d2,1+\alpha],$ there exists  a constant $c_2>0$ such that
\begin{align*}
|Er_2|
&\le c_2 \Big(\e^{-\ff 1 \vv t}\|\nn\psi\|_\infty + \ff{\gg(\vv) (1 +\kk_1)}{2\ss{\vv  }}\|\psi\|_\infty\Big) \vv^{-1}\E \int_0^t  \Big[s^{-\ff {\aa_1} 2}
+  \vv^{-1} \big(1+\|u_s^N(x)\|_{\dot\H^{{\beta}{}}}^{2m}\big)\\\nonumber
&+\vv^{- ({md+2})m-1}\Big(1+\|u_s^N(x)\|_{\dot\H^1}^{(2m+1)2m}\Big) \Big]    \|(1-\pi_N)f(u^N_s( x))\|_{\dot \H^{-\alpha_1}} \text{d} s.
\end{align*}

Using \eqref{*2'}, \eqref{*2''} and \eqref{add-condition1}, there exist constants $c_3>0,c_4>0$ such that \begin{equation}\begin{split}
|Er_2|&\le c_3(1+t)\Big(\e^{-\ff 1 \vv t}\|\nn\psi\|_\infty + \ff{\gg(\vv) (1 +\kk_1)}{2\ss{\vv  }}\|\psi\|_\infty\Big) \lambda_N^{-\frac {\alpha_1} 2-\frac {\beta} 2} \\\nonumber
&\quad \times \Big(\vv^{-2}+\vv^{-2}\|x\|_{\dot \H^{\beta}}^{2m}+\vv^{- {(md+2)}m-2-m(2m+1)}\|x\|_{\dot \H^1}^{2m(2m+1)}\Big) \\\nonumber
&\quad\times \sup_{s\in [0,t]}\sqrt{\E[\|f(u_s^N(x))\|_{\dot \H^{\beta}}^2]} \\
&\le c_4(1+t)\Big(\e^{-\ff 1 \vv t}\|\nn\psi\|_\infty + \ff{\gg(\vv) (1 +\kk_1)}{2\ss{\vv  }}\|\psi\|_\infty\Big)  \lambda_N^{-\frac {\alpha_1} 2-\frac {\beta} 2} \\
&\quad \times
\Big(\vv^{-2}+\vv^{-2}\|x\|_{\dot \H^{\beta}}^{2m}+\vv^{- {(md+2)}m-2-m(2m+1)}\|x\|_{\dot \H^1}^{2m(2m+1)}\Big) \\\nonumber
&\quad \times \Big(1+\|x\|_{\dot \H^{\beta}}^{2m+1}+\vv^{- \ff {(md+2)}2(2m+1)}\|x\|_{\dot \H^1}^{(2m+1)^2}\Big).
\end{split}
\end{equation}
Combining \eqref{trun-err} and the estimates of $Er_1$-$Er_2$, we complete the proof.
\end{proof}

\subsection{Proof of Lemma \ref{second-finite}}

\begin{proof}[Proof of Lemma \ref{second-finite}]
By introducing $I_{s,t}$ in \eqref{02}, the solution of \eqref{var-2} satisfies
\begin{align}\label{mild-second}
 \zeta_N^{h_1,h_2}(t)= -\vv^{-1} \int_0^t I_{s,t} \pi_N f''(u^N_s)\cdot \left(\eta_{N}^{h_1}(s),\eta_{N}^{h_2}(s)\right)ds.
\end{align}
Similar to the proof of \eqref{GR}, one can prove that for any $\alpha_1\in (0,2),$ there exists a constant $c_1>0$ such that
\begin{align*}
&\|I_{s,t}(h)\|_\H=\|h_{s,t}\|_\H\\
&\le \|h_{s,t}-e^{-A(t-s)}h\|_{\H}+ \|e^{-A(t-s)}h\|_{\H}\\
&\le c_1\|h\|_{\dot\H^{-\aa_1}}\bigg[(t-s)^{-\ff{\aa_1} 2}
+ \ff{\e^{\ff{\kk_1}\vv (t-s)}}\vv \int_s^t \e^{-\ff{\ll_1}2 (r-s)}\Big(1+\|u_r^N(x)\|_{L^{\infty} (\mu)}^{2m}\Big)\d r\bigg].
\end{align*}
Combining  this with the property of Gamma function and \eqref{*2''}, we have that
for any $\beta\in (\frac d2,1+\alpha]$ and $\alpha_1\in (0,2),$ there exists a constant $c_2>0$ such that
\begin{align}\label{ist2}
&\E[ \|I_{s,t}(h)\|_\H| \F_s]\le c_2\|h\|_{\dot\H^{-\aa_1}}\bigg[(t-s)^{-\ff{\aa_1} 2}
\\\nonumber
&+ {\e^{\ff{\kk_1}\vv (t-s)}} \vv^{-1} \Big(1+\|u_s^N(x)\|_{\dot \H^{\beta}}^{2m}+ \vv^{ - m(md+2)}(1+\|u_s^N(x)\|_{\H^1}^{(2m+1)2m})\Big)\bigg].
\end{align}
By the duality and \eqref{Sob1}, it holds that for some $c>0,$
\begin{align*}
\|u\|_{\dot \H^{-\beta}}\le c\|u\|_{L^{1}(\mu)}, \;\forall \; u\; \in L^1(\mu).
\end{align*}

Combining this with \eqref{ist2}, \eqref{mild-second} and \eqref{F5}, we obtain that for any $\beta\in (\frac d2,1+\alpha],$ there exists constants $c_3>0,c_4>0$ such that
\begin{align*}
&\E \|\zeta_N^{h_1,h_2}(t)\|_{\H}\\&= \vv^{-1} \E \Big\| \int_0^t I_{s,t} \pi^N f''(u^N_s)\cdot \left(\eta_{N}^{h_1}(s),\eta_{N}^{h_2}(s)\right)ds\Big\|_{\H}\\
&\le  \vv^{-1} \E \int_0^t \E \Big[\Big\| I_{s,t} \pi^N f''(u^N_s)\cdot \left(\eta_{N}^{h_1}(s),\eta_{N}^{h_2}(s)\right)\Big\|_{\H} \; \Big| \; \mathcal F_s\Big]ds\\
&\le c_3\vv^{-1} \E \int_0^t \|f''(u^N_s)\cdot \left(\eta_{N}^{h_1}(s),\eta_{N}^{h_2}(s)\right)\Big\|_{\dot \H^{-\beta}} \bigg[(t-s)^{-\ff {\beta} 2}\\
&
+ {\e^{\ff{\kk_1}\vv (t-s)}} \vv^{-1} \Big(1+\|u_s^N(x)\|_{\dot \H^\beta}^{2m}+ \vv^{ - m(md+2)}(1+\|u_s^N(x)\|_{\dot \H^1}^{2m(2m+1)})\Big)\bigg] ds\\
&\le c_4 \vv^{-1} \E \int_0^t (1+\|u_{s}^N(x)\|_{L^{\infty}(\mu)}^{\max(2m-1,0)})\|\eta_{N}^{h_1}(s)\|_{\H}\|\eta_{N}^{h_2}(s)\|_{\H}\bigg[(t-s)^{-\ff {\beta} 2}\\
&
+ {\e^{\ff{\kk_1}\vv (t-s)}} \vv^{-1} \Big(1+\|u_s^N(x)\|_{\dot \H^\beta}^{2m}+ \vv^{ - m(md+2)}(1+\|u_s^N(x)\|_{\dot \H^1}^{2m(2m+1)})\Big)\bigg] ds.
\end{align*}
Applying \eqref{04}, the H\"older's inequality, and using the moment estimates \eqref{*2'}-\eqref{*2''} and \eqref{I1}-\eqref{I2}, one further obtain that for some $c_5>0,c_6>0$,
\begin{equation}\begin{split}\label{second-d1}
&\E \|\zeta_N^{h_1,h_2}(t)\|_{\H}\le c_5 e^{3\frac {\kappa_1}\vv t}\vv^{-1} \E \int_0^t (1+\|u_{s}^N(x)\|_{L^{\infty}(\mu)}^{\max(2m-1,0)})\|h_1\|_{\H}\|h_2\|_{\H} \\
&\quad\quad\times \bigg[(t-s)^{-\ff {\beta} 2}
+ \vv^{-1} \Big(1+\|u_s^N(x)\|_{\dot \H^\beta}^{2m}+\vv^{-m(md+2)}\|u_s^N(x)\|_{\dot \H^1}^{2m(2m+1)}\Big)\bigg] ds \\
&\quad\quad \le c_6(1+t) e^{3\frac {\kappa_1}\vv t}\vv^{-2}\|h_1\|_{\H}\|h_2\|_{\H}\\
&\quad\quad \times \Big(1+\|x\|_{\dot\H^{\beta}}+ \vv^{-\ff{md+2} 2}(1+\|x\|_{\dot \H^1})^{2m+1}\Big)^{2m+\max(2m-1,0)}.
\end{split}\end{equation}

Now, we are in a position to show the finiteness of $\nabla^2 P_t^N\psi.$
According to \eqref{rep-kol}, it follows that
\begin{align*}
&\Big|\nabla^2 P_t^N\psi(x) \cdot (h_1,h_2)\Big|\\
&\le \Big|\E^{x} \left[\nabla \psi(u^{N}_t)\cdot \zeta_N^{h_1,h_2}(t)\right]\Big|+\Big|\E^{x} \left[\nabla^2\psi({u^N_t})\cdot (\eta_N^{h_1}(t),\eta_N^{h_2}(t))\right]\Big|\\
&\le \|\nabla \psi\|_{\infty} \E^x \|\zeta_N^{h_1,h_2}(t)\|_\H+ \|\nabla^2 \psi\|_{\infty}\E^x \|\eta_N^{h_1}(t)\|_\H \|\eta_N^{h_2}(t)\|_\H.
\end{align*}
Using \eqref{04} and the estimates of $\E^x \|\zeta_N^{h_1,h_2}(t)\|_\H$ \eqref{second-d1}, we obtain that
\begin{align*}
&\Big|\nabla^2 P_t^N\psi(x) \cdot (h_1,h_2)\Big|\\
&\le  \|\nabla^2 \psi\|_{\infty} e^{2\frac {\kappa_1}{\vv}t}\|h_1\|_\H\|h_2\|_\H \\
&+ c_6 \|\nabla \psi\|_{\infty} (1+t) e^{3\frac {\kappa_1}\vv t}\|h_1\|_{\H}\|h_2\|_{\H}\\
\nonumber
&\quad \times \vv^{-2}\Big(1+\|x\|_{\dot\H^{\beta}}+ \vv^{-\ff{md+2} 2}(1+\|x\|_{\dot \H^1})^{2m+1}\Big)^{2m+\max(2m-1,0)}.
\end{align*}
This implies the finiteness of $\nabla^2 P_t^N\psi$ under the $\dot\H^{0, 0}$-norm.

\end{proof}

\end{document}